\documentclass[10pt,portrait,reqno,11pt]{amsart}%
\usepackage{amssymb,amsthm,amsmath, amsfonts, float}
\usepackage{graphicx}
\usepackage{multirow}
\usepackage{subfig}
\usepackage{amsmath}
\usepackage{amsfonts}
\usepackage{lscape}
\usepackage{amssymb}%
\providecommand{\U}[1]{\protect\rule{.1in}{.1in}}
\newtheorem{theorem}{Theorem}
\theoremstyle{plain}

\newtheorem{example}{Example}

\newtheorem{proposition}{Proposition}

\numberwithin{equation}{section}
\begin{document}
\title[Canal Hypersurfaces Gen. by (Pseudo-Partially) Null Curves in $E_{1}^{4}$]{Canal Hypersurfaces Generated by Pseudo Null, Partially Null and Null Curves
in Lorentz-Minkowski 4-Space}
\author[M. Alt\i n, A. Kazan and D.W. Yoon]{\bfseries Mustafa Alt\i n$^{1\ast}$, Ahmet Kazan$^{2}$ and Dae Won Yoon$^{3}$}
\address{ \\
$^{1}$Technical Sciences Vocational School, Bing\"{o}l University, Bing\"{o}l,
Turkey \\
 \\
$^{2}$Department of Computer Technologies, Do\u{g}an\c{s}ehir Vahap
K\"{u}\c{c}\"{u}k Vocational School, Malatya Turgut \"{O}zal University,
Malatya, Turkey \\
 \\
$^{3}$Department of Mathematics Education and RINS, Gyeongsang National
University, Jinju 52828, Republic of Korea \\
 \\
$^{\ast}$Corresponding author: maltin@bingol.edu.tr}

\begin{abstract}
In this paper, we obtain the parametric expressions of the canal hypersurfaces
that are formed as the envelope of a family of pseudo hyperspheres or pseudo
hyperbolic hyperspheres whose centers lie on a pseudo null, partially null or
null curves in $E_{1}^{4}$ and give their some geometric invariants such as
unit normal vector fields, Gaussian curvatures and mean curvatures. Also, we
construct some examples for these canal hypersurfaces and finally, we give
some characterizations for tubular hypersurfaces in $E_{1}^{4}$.

\end{abstract}
\maketitle


\section{\textbf{GENERAL INFORMATION AND\ BASIC\ CONCEPTS}}

The class of surfaces formed by sweeping a sphere is called canal surfaces and
they have been investigated by Monge in 1850. So, one can see a canal surface
as the envelope of a moving sphere with varying radius, defined by the
trajectory $\gamma(s)$ of its centers and a radius function $r(s)$. The canal
surface is called a tubular surface or pipe surface if the radius function is
constant. After defining canal and tubular surfaces, many mathematicians have
studied their different geometric characterizations and these surfaces have
been applied to many fields, such as the solid and the surface modeling for
CAD/CAM, construction of blending surfaces, shape re-construction. Also, they
are useful to represent various objects such as pipe, hose, rope or intestine
of a body \cite{Karacany}, \cite{Ucum}. For some results about different
geometric characterizations of canal and tubular (hyper)surfaces in different
three, four or higher dimensional spaces, we refer to \cite{Mahmut},
\cite{Aslan}, \cite{Yusuf}, \cite{Fu}, \cite{Garcia}, \cite{Hartman},
\cite{Izumiya}, \cite{Karacan}, \cite{Karacan2}, \cite{Karacan3},
\cite{Sezai}, \cite{Kim}, \cite{Krivos}, \cite{Maekawa}, \cite{Peter},
\cite{Ro}, \cite{Ucum}, \cite{Xu}, \cite{Kucuk}, and etc.

On the other hand, the extrinsic differential geometry of submanifolds in
Lorentz-Minkowski 4-space $E_{1}^{4}$ is of special interest in Relativity
Theory. So, many studies about curves and (hyper)surfaces have been done in
this space and this motivated us to construct the canal hypersurfaces using
pseudo null, partially null and null curves in $E_{1}^{4}$. Here, we'll recall
some basic concepts for curves and hypersurfaces in $E_{1}^{4}$.

Let $\overrightarrow{u}=(u_{1},u_{2},u_{3},u_{4})$, $\overrightarrow{v}%
=(v_{1},v_{2},v_{3},v_{4})$ and $\overrightarrow{w}=(w_{1},w_{2},w_{3},w_{4})$
be three vectors in $E_{1}^{4}$. Then the inner product of two vectors and
vector product of three vectors in $E_{1}^{4}$ are given by%
\begin{equation}
\left\langle \overrightarrow{u},\overrightarrow{v}\right\rangle =-u_{1}%
v_{1}+u_{2}v_{2}+u_{3}v_{3}+u_{4}v_{4} \label{yy1}%
\end{equation}
and
\begin{equation}
\overrightarrow{u}\times\overrightarrow{v}\times\overrightarrow{w}=\det\left[
\begin{array}
[c]{cccc}%
-e_{1} & e_{2} & e_{3} & e_{4}\\
u_{1} & u_{2} & u_{3} & u_{4}\\
v_{1} & v_{2} & v_{3} & v_{4}\\
w_{1} & w_{2} & w_{3} & w_{4}%
\end{array}
\right]  , \label{yy2}%
\end{equation}
respectively.

A vector $\overrightarrow{u}\in E_{1}^{4}-\{0\}$ is called spacelike if
$\left\langle \overrightarrow{u},\overrightarrow{u}\right\rangle >0$; timelike
if $\left\langle \overrightarrow{u},\overrightarrow{u}\right\rangle <0$ and
lightlike (null) if $\left\langle \overrightarrow{u},\overrightarrow{u}%
\right\rangle =0$. In particular, the vector $\overrightarrow{u}=0$ is
spacelike. Also, the norm of the vector $\overrightarrow{u}$ is $\left\Vert
\overrightarrow{u}\right\Vert =\sqrt{\left\vert \left\langle
\overrightarrow{u},\overrightarrow{u}\right\rangle \right\vert }$. A curve
$\gamma(s)$ in $E_{1}^{4}$ is spacelike, timelike or lightlike (null), if all
its velocity vectors $\gamma^{\prime}(s)$ are spacelike, timelike or
lightlike, respectively and a non-null (i.e. timelike or spacelike) curve has
unit speed if $\left\langle \gamma^{\prime},\gamma^{\prime}\right\rangle
=\mp1$. If $\left\langle \gamma^{\prime\prime}(s),\gamma^{\prime\prime
}(s)\right\rangle =1$, then the null curve $\gamma$ is parametrized by
arclength function $s$. Also, if the principal normal vector or the binormal
vector of a spacelike curve $\gamma(s)$ in $E_{1}^{4}$ is null, then the
spacelike curve $\gamma(s)$ is called a pseudo null curve or a partially null
curve (for detail, one can see \cite{Bonnor1}, \cite{Bonnor2}, \cite{Kuhnel},
\cite{Oneil}, and etc.)

If $\{F_{1},F_{2},F_{3},F_{4}\}$ is the moving Frenet frame along a curve
$\gamma(s)$ in $E_{1}^{4}$ consisting of the unit tangent vector field,
principal normal vector field, binormal vector field and trinormal vector
field, then we have the following cases according to the causal character of
$\gamma$ (see \cite{Bonnor1}, \cite{Bonnor2}, \cite{Kazim}, \cite{Walrave},
and etc):

\textbf{i)} If the curve $\gamma(s)$ is pseudo null, then the Frenet formulas
are%
\begin{equation}
\left.
\begin{array}
[c]{l}%
F_{1}^{\prime}=k_{1}F_{2},\\
F_{2}^{\prime}=k_{2}F_{3},\\
F_{3}^{\prime}=k_{3}F_{2}-k_{2}F_{4},\\
F_{4}^{\prime}=-k_{1}F_{1}-k_{3}F_{3},
\end{array}
\right\}  \label{PsnullFrenet}%
\end{equation}
and we have%
\begin{equation}
\left.
\begin{array}
[c]{l}%
\left\langle F_{1},F_{1}\right\rangle =\left\langle F_{3},F_{3}\right\rangle
=1,\text{ }\left\langle F_{2},F_{2}\right\rangle =\left\langle F_{4}%
,F_{4}\right\rangle =0,\\
\left\langle F_{2},F_{4}\right\rangle =1,\text{ }\left\langle F_{1}%
,F_{2}\right\rangle =\left\langle F_{1},F_{3}\right\rangle =\left\langle
F_{1},F_{4}\right\rangle =\left\langle F_{2},F_{3}\right\rangle =\left\langle
F_{3},F_{4}\right\rangle =0.
\end{array}
\right\}  \label{PsnullFi ler}%
\end{equation}

\textbf{ii)} If the curve $\gamma(s)$ is partially null, then the Frenet
formulas are%
\begin{equation}
\left.
\begin{array}
[c]{l}%
F_{1}^{\prime}=k_{1}F_{2},\\
F_{2}^{\prime}=-k_{1}F_{1}+k_{2}F_{3},\\
F_{3}^{\prime}=k_{3}F_{3},\\
F_{4}^{\prime}=-k_{2}F_{2}-k_{3}F_{4},
\end{array}
\right\}  \label{PartnullFrenet}%
\end{equation}
and we have%
\begin{equation}
\left.
\begin{array}
[c]{l}%
\left\langle F_{1},F_{1}\right\rangle =\left\langle F_{2},F_{2}\right\rangle
=1,\text{ }\left\langle F_{3},F_{3}\right\rangle =\left\langle F_{4}%
,F_{4}\right\rangle =0,\\
\left\langle F_{3},F_{4}\right\rangle =1,\text{ }\left\langle F_{1}%
,F_{2}\right\rangle =\left\langle F_{1},F_{3}\right\rangle =\left\langle
F_{1},F_{4}\right\rangle =\left\langle F_{2},F_{3}\right\rangle =\left\langle
F_{2},F_{4}\right\rangle =0.
\end{array}
\right\}  \label{PartnullFi ler}%
\end{equation}

\textbf{iii)} If the curve $\gamma(s)$ is null, then the Frenet formulas are%
\begin{equation}
\left.
\begin{array}
[c]{l}%
F_{1}^{\prime}=k_{1}F_{2},\\
F_{2}^{\prime}=k_{2}F_{1}-k_{1}F_{3},\\
F_{3}^{\prime}=-k_{2}F_{2}+k_{3}F_{4},\\
F_{4}^{\prime}=-k_{3}F_{1},
\end{array}
\right\}  \label{nullFrenet}%
\end{equation}
and we have%
\begin{equation}
\left.
\begin{array}
[c]{l}%
\left\langle F_{2},F_{2}\right\rangle =\left\langle F_{4},F_{4}\right\rangle
=1,\text{ }\left\langle F_{1},F_{1}\right\rangle =\left\langle F_{3}%
,F_{3}\right\rangle =0,\\
\left\langle F_{1},F_{3}\right\rangle =1,\text{ }\left\langle F_{1}%
,F_{2}\right\rangle =\left\langle F_{1},F_{4}\right\rangle =\left\langle
F_{2},F_{3}\right\rangle =\left\langle F_{2},F_{4}\right\rangle =\left\langle
F_{3},F_{4}\right\rangle =0.
\end{array}
\right\}  \label{nullFi ler}%
\end{equation}

Here, the first curvature $k_{1}(s)=0$, if $\gamma$ is a straight line and
$k_{1}(s)=1$ in all other cases for pseudo null and null curves. So, the
pseudo null and null curves have two curvatures $k_{2}(s)$ and $k_{3}(s)$.
Also, the third curvature $k_{3}(s)=0$ for each $s$ for a partially null curve
and thus, a partially null curve has two curvatures $k_{1}(s)$ and $k_{2}(s)$.

Furthermore, if $p$ is a fixed point in $E_{1}^{4}$ and $r$ is a positive
constant, then the pseudo-Riemannian hypersphere, pseudo-Riemannian hyperbolic
space and pseudo-Riemannian null hypercone are defined by%
\begin{equation}
\left.
\begin{array}
[c]{l}%
S_{1}^{3}(p,r)=\{u\in E_{1}^{4}:\left\langle u-p,u-p\right\rangle =r^{2}\},\\
H_{0}^{3}(p,r)=\{u\in E_{1}^{4}:\left\langle u-p,u-p\right\rangle =-r^{2}\},\\
Q_{1}^{3}=\{u\in E_{1}^{4}:\left\langle u-p,u-p\right\rangle =0\},
\end{array}
\right\}  \label{y1y}%
\end{equation}
respectively.

In the present study, we construct the canal hypersurfaces in $E_{1}^{4}$ as
the envelope of a family of pseudo hyperspheres or pseudo hyperbolic
hyperspheres whose centers lie on a pseudo null, partially null or null curve.

Furthermore, the differential geometry of different types of (hyper)surfaces
in 4-dimensional spaces has been a popular topic for geometers, recently
(\cite{Altinyy}, \cite{Altin4}, \cite{Altin2}, \cite{Altin3}, \cite{Aydin},
\cite{Altin}, \cite{Altiny}, \cite{Kazantek}, \cite{Kisi}, and etc). In this
context, let $\Omega$ be a hypersurface in $E_{1}^{4}$ given by
\begin{align}
\Omega:U\subset E^{3}  &  \longrightarrow E_{1}^{4}\label{yy3}\\
(u_{1},u_{2},u_{3})  &  \longrightarrow\Omega(u_{1},u_{2},u_{3})=(\Omega
_{1}(u_{1},u_{2},u_{3}),\Omega_{2}(u_{1},u_{2},u_{3}),\Omega_{3}(u_{1}%
,u_{2},u_{3}),\Omega_{4}(u_{1},u_{2},u_{3})).\nonumber
\end{align}
Then the Gauss map (i.e., the unit normal vector field), the matrix forms of
the first and second fundamental forms are%
\begin{equation}
N_{\Omega}=\frac{\Omega_{u_{1}}\times\Omega_{u_{2}}\times\Omega_{u_{3}}%
}{\left\Vert \Omega_{u_{1}}\times\Omega_{u_{2}}\times\Omega_{u_{3}}\right\Vert
},\text{ }[g_{ij}]=\left[
\begin{array}
[c]{ccc}%
g_{11} & g_{12} & g_{13}\\
g_{21} & g_{22} & g_{23}\\
g_{31} & g_{32} & g_{33}%
\end{array}
\right]  \text{ and }[h_{ij}]=\left[
\begin{array}
[c]{ccc}%
h_{11} & h_{12} & h_{13}\\
h_{21} & h_{22} & h_{23}\\
h_{31} & h_{32} & h_{33}%
\end{array}
\right]  , \label{4y}%
\end{equation}
respectively. Here $g_{ij}=\left\langle \Omega_{u_{i}},\Omega_{u_{j}%
}\right\rangle ,$ $h_{ij}=\left\langle \Omega_{u_{i}u_{j}},N_{\Omega
}\right\rangle ,$ $\Omega_{u_{i}}=\frac{\partial\Omega}{\partial u_{i}},$
$\Omega_{u_{i}u_{j}}=\frac{\partial^{2}\Omega}{\partial u_{i}u_{j}},$
$i,j\in\{1,2,3\}.$ Also, the matrix of shape operator of the hypersurface
(\ref{yy3}) is%
\begin{equation}
S=[a_{ij}]=[g^{ij}].[h_{ij}], \label{7yyy}%
\end{equation}
where $[g^{ij}]$ is the inverse matrix of $[g_{ij}]$. With the aid of
(\ref{4y}) and (\ref{7yyy}), the Gaussian curvature and mean curvature of a
hypersurface in $E_{1}^{4}$ are given by%
\begin{equation}
K=\varepsilon\frac{\det[h_{ij}]}{\det[g_{ij}]}\text{ \ \ and \ }3\varepsilon
H=tr(S), \label{yy4}%
\end{equation}
respectively. Here, $\varepsilon=\left\langle N_{\Omega},N_{\Omega
}\right\rangle .$ We say that a hypersurface is flat or minimal, if it has
zero Gaussian or zero mean curvature, respectively. For more details about
hypersurfaces in $E_{1}^{4},$ we refer to \cite{Guler1}, \cite{Lee} and etc.

\section{\textbf{CANAL HYPERSURFACES GENERATED BY PSEUDO NULL, PARTIALLY NULL
AND NULL CURVES IN }$E_{1}^{4}$}

In this section, firstly we construct the canal hypersurfaces that are formed
as the envelope of a family of pseudo hyperspheres or pseudo hyperbolic
hyperspheres whose centers lie on a pseudo null curve, partially null curve or
null curve in $E_{1}^{4}$. After that, we obtain some important geometric
invariants such as unit normal vector fields, Gaussian curvatures and mean
curvatures of these canal hypersurfaces, separately.

\begin{theorem}
\textit{The canal hypersurfaces that are formed as the envelope of a family of
pseudo hyperspheres whose centers lie on a pseudo null curve }$\gamma(s)$ with
Frenet vector fields $F_{i},$ $i\in\{1,2,3,4\},$\textit{ in }$E_{1}^{4}$
\textit{can be parametrized by}%
\begin{equation}
\left.
\begin{array}
[c]{l}%
\overset{psd-n}{C_{1}}(s,t,w)=\gamma(s)-r(s)r^{\prime}(s)F_{1}(s)\\
\text{ \ \ }\mp r(s)\sqrt{1-r^{\prime}(s)^{2}}\left(  g(t,w)\sin
(f(t,w))F_{2}(s)+\cos(f(t,w))F_{3}(s)+\frac{\sin(f(t,w))}{2g(t,w)}%
F_{4}(s)\right)  ,\\
\overset{psd-n}{C_{2}}(s,t,w)=\gamma(s)-r(s)r^{\prime}(s)F_{1}(s)\\
\text{ \ \ }\mp r(s)\sqrt{1-r^{\prime}(s)^{2}}\left(  g(t,w)\cos
(f(t,w))F_{2}(s)+\sin(f(t,w))F_{3}(s)+\frac{\cos(f(t,w))}{2g(t,w)}%
F_{4}(s)\right)  ,\\
\overset{psd-n}{C_{3}}(s,t,w)=\gamma(s)-r(s)r^{\prime}(s)F_{1}(s)\\
\text{ \ \ }\mp r(s)\sqrt{1-r^{\prime}(s)^{2}}\left(  g(t,w)\sinh
(f(t,w))F_{2}(s)+\cosh(f(t,w))F_{3}(s)-\frac{\sinh(f(t,w))}{2g(t,w)}%
F_{4}(s)\right)  ,\\
\overset{psd-n}{C_{4}}(s,t,w)=\gamma(s)-r(s)r^{\prime}(s)F_{1}(s)\\
\text{ \ \ }\mp r(s)\sqrt{r^{\prime}(s)^{2}-1}\left(  g(t,w)\cosh
(f(t,w))F_{2}(s)+\sinh(f(t,w))F_{3}(s)-\frac{\cosh(f(t,w))}{2g(t,w)}%
F_{4}(s)\right)  ,
\end{array}
\right\}  \label{canal1}%
\end{equation}
where we suppose $r^{\prime}(s)^{2}<1$ for canal hypersurfaces
$\overset{psd-n}{C_{1}},\overset{psd-n}{C_{2}},\overset{psd-n}{C_{3}}$ and
$r^{\prime}(s)^{2}>1$ for canal hypersurface $\overset{psd-n}{C_{4}}$.

\textit{Furthermore, the canal hypersurfaces that are formed as the envelope
of a family of pseudo hyperbolic hyperspheres whose centers lie on a pseudo
null curve }$\gamma(s)$ with Frenet vector fields $F_{i},$ $i\in\{1,2,3,4\},$
\textit{in }$E_{1}^{4}$ \textit{can be parametrized by}%
\begin{align}
&  \overset{psd-n}{C_{5}}(s,t,w)=\gamma(s)+r(s)r^{\prime}(s)F_{1}%
(s)\label{canal2}\\
&  \text{ \ \ }\mp r(s)\sqrt{1+r^{\prime}(s)^{2}}\left(  g(t,w)\cosh
(f(t,w))F_{2}(s)+\sinh(f(t,w))F_{3}(s)-\frac{\cosh(f(t,w))}{2g(t,w)}%
F_{4}(s)\right)  .\nonumber
\end{align}
Here, the canal hypersurfaces $\overset{psd-n}{C_{1}},$ $\overset{psd-n}{C_{2}%
},$ $\overset{psd-n}{C_{3}},$ $\overset{psd-n}{C_{4}}$ are timelike and the
canal hypersurface $\overset{psd-n}{C_{5}}$ is spacelike.
\end{theorem}

\begin{proof}
Let the center curve $\gamma:I\subseteq%
\mathbb{R}
\rightarrow E_{1}^{4}$ be a pseudo null curve with non-zero curvature with
Frenet vector fields $F_{1}(s),$ $F_{2}(s),$ $F_{3}(s),$ $F_{4}(s)$ called
unit tangent, principal normal, binormal and trinormal vectors of $\gamma(s)$,
respectively. Then, the parametrization of the envelope of pseudo hyperspheres
(or pseudo hyperbolic hyperspheres) defining the canal hypersurfaces
$\overset{{\small psd-n}}{C}(s,t,w)$ in $E_{1}^{4}$ can be given by%
\begin{equation}
\overset{{\small psd-n}}{C}(s,t,w)-\gamma(s)=a_{1}(s,t,w)F_{1}(s)+a_{2}%
(s,t,w)F_{2}(s)+a_{3}(s,t,w)F_{3}(s)+a_{4}(s,t,w)F_{4}(s), \label{1}%
\end{equation}
where $a_{i}(s,t,w)$ are differentiable functions of $s,t,w$ on the interval
$I$. Furthermore, since $\overset{{\small psd-n}}{C}(s,t,w)$ lies on the
pseudo hyperspheres ($\lambda=1$) (or pseudo hyperbolic hyperspheres
($\lambda=-1$)), we have%
\begin{equation}
g(\overset{{\small psd-n}}{C}(s,t,w)-\gamma(s),\overset{{\small psd-n}%
}{C}(s,t,w)-\gamma(s))=\lambda r^{2}(s) \label{2}%
\end{equation}
which leads to from (\ref{1}) and (\ref{PsnullFi ler}) that%
\begin{equation}
a_{1}^{2}+a_{3}^{2}+2a_{2}a_{4}=\lambda r^{2} \label{3}%
\end{equation}
and%
\begin{equation}
a_{1}a_{1_{s}}+a_{3}a_{3_{s}}+a_{4}a_{2_{s}}+a_{2}a_{4_{s}}=\lambda rr_{s},
\label{4}%
\end{equation}
where $r(s)$ is the radius function; $r=r(s),$ $r_{s}=\frac{dr(s)}{ds},$
$a_{i}=a_{i}(s,t,w),$ $a_{i_{s}}=\frac{\partial a_{i}(s,t,w)}{\partial s}$.

So, differentiating (\ref{1}) with respect to $s$ and using the Frenet formula
(\ref{PsnullFrenet}), we get%
\begin{align}
(\overset{psd-n}{C})_{s}  &  =\left(  1-a_{4}k_{1}+a_{1_{s}}\right)
F_{1}+\left(  a_{1}k_{1}+a_{3}k_{3}+a_{2_{s}}\right)  F_{2}\label{5'}\\
&  +\left(  a_{2}k_{2}-a_{4}k_{3}+a_{3_{s}}\right)  F_{3}+\left(  -a_{3}%
k_{2}+a_{4_{s}}\right)  F_{4},\nonumber
\end{align}
where $(\overset{psd-n}{C})_{s}=\frac{\partial\left(  \overset{psd-n}{C}%
(s,t,w)\right)  }{\partial s}$. Furthermore, $\overset{psd-n}{C}%
(s,t,w)-\gamma(s)$ is a normal vector to the canal hypersurfaces, which
implies that%
\begin{equation}
g(\overset{psd-n}{C}(s,t,w)-\gamma(s),(\overset{psd-n}{C})_{s}(s,t,w))=0
\label{6}%
\end{equation}
and so, from (\ref{1}), (\ref{5'}) and (\ref{6}) we have%
\begin{equation}
\text{ }\left(
\begin{array}
[c]{l}%
a_{1}\left(  1-a_{4}k_{1}+a_{1_{s}}\right)  +a_{2}\left(  -a_{3}k_{2}%
+a_{4_{s}}\right) \\
+a_{3}\left(  a_{2}k_{2}-a_{4}k_{3}+a_{3_{s}}\right)  +a_{4}\left(  a_{1}%
k_{1}+a_{3}k_{3}+a_{2_{s}}\right)
\end{array}
\right)  =0. \label{6y}%
\end{equation}
Using (\ref{4}) in (\ref{6y}), we get
\begin{equation}
a_{1}=-\lambda rr_{s}. \label{7y}%
\end{equation}
Hence, using (\ref{7y}) in (\ref{3}), we reach that%
\begin{equation}
{\small a}_{3}^{2}{\small +}2{\small a}_{2}a_{4}{\small =\lambda r}%
^{2}{\small (}1-\lambda{\small r}_{s}^{2}{\small ).} \label{8}%
\end{equation}
Therefore from (\ref{7y}) and (\ref{8}), the canal hypersurfaces
$\overset{psd-n}{C}(s,t,w)$ that are formed as the envelope of a family of
pseudo hyperspheres or pseudo hyperbolic hyperspheres whose centers lie on a
pseudo null curve in $E_{1}^{4}$ can be parametrized by (\ref{canal1}) or
(\ref{canal2}), respectively.
\end{proof}

Here we must note that, from now on we will state $r^{\prime}=r_{s}(s),$
$r^{\prime\prime}=\frac{d^{2}r(s)}{ds^{2}},$ $f=f(t,w)$, $g=g(t,w),$ $\sin
f=\sin(f(t,w)),$ $f_{t}=\frac{\partial f(t,w)}{\partial t},$ $f_{tt}%
=\frac{\partial f^{2}(t,w)}{\partial t^{2}},$ and so on.

\begin{theorem}
\label{Teopseudonull}\textit{The Gaussian and mean curvatures of the canal
hypersurfaces }$\overset{psd-n}{C_{i}},$ $i\in\{1,2,3,4,5\},$ given by
(\ref{canal1}) and (\ref{canal2}) in $E_{1}^{4}$ are%
\begin{equation}
\left.
\begin{array}
[c]{l}%
\left.
\begin{array}
[c]{l}%
K_{\overset{psd-n}{C_{1}}}=\frac{-r\left(  1-{\small r}^{\prime2}\right)
k_{1}^{2}\sin^{2}f{\small +4}r^{\prime\prime}\left(  1-r^{\prime2}%
-rr^{\prime\prime}\right)  {\small g}^{2}+2\sqrt{1-{\small r}^{\prime2}%
}\left(  1-r^{\prime2}-2rr^{\prime\prime}\right)  k_{1}g\sin f}{r^{2}\left(
r\sqrt{1-{\small r}^{\prime2}}k_{1}\sin f-2\left(  1-r^{\prime2}%
-rr^{\prime\prime}\right)  g\right)  ^{2}}\\
H_{\overset{psd-n}{C_{1}}}=\frac{2r\left(  1-{\small r}^{\prime2}\right)
^{3/2}k_{1}g\sin f+3r^{2}\left(  1-{\small r}^{\prime2}\right)  k_{1}^{2}%
\sin^{2}f-4\left(  1-r^{\prime2}-rr^{\prime\prime}\right)  \left(
2-2r^{\prime2}-3rr^{\prime\prime}\right)  g^{2}}{3\left(  -r^{3}\left(
1-{\small r}^{\prime2}\right)  k_{1}^{2}\sin^{2}f+4r\left(  1-r^{\prime
2}-rr^{\prime\prime}\right)  ^{2}g^{2}\right)  }%
\end{array}
\right\}  ,\\
\\
\left.
\begin{array}
[c]{l}%
K_{\overset{psd-n}{C_{2}}}=\frac{r\left(  1-{\small r}^{\prime2}\right)
k_{1}^{2}\cos^{2}f{\small -4}r^{\prime\prime}\left(  1-r^{\prime2}%
-rr^{\prime\prime}\right)  {\small g}^{2}-2\sqrt{1-{\small r}^{\prime2}%
}\left(  1-r^{\prime2}-2rr^{\prime\prime}\right)  k_{1}g\cos f}{r^{2}\left(
r\sqrt{1-{\small r}^{\prime2}}k_{1}\cos f-2\left(  1-r^{\prime2}%
-rr^{\prime\prime}\right)  g\right)  ^{2}}\\
H_{\overset{psd-n}{C_{2}}}=\frac{-2r\left(  1-{\small r}^{\prime2}\right)
^{3/2}k_{1}g\cos f-3r^{2}\left(  1-{\small r}^{\prime2}\right)  k_{1}^{2}%
\cos^{2}f+4\left(  1-r^{\prime2}-rr^{\prime\prime}\right)  \left(
2-2r^{\prime2}-3rr^{\prime\prime}\right)  g^{2}}{3\left(  -r^{3}\left(
1-{\small r}^{\prime2}\right)  k_{1}^{2}\cos^{2}f+4r\left(  1-r^{\prime
2}-rr^{\prime\prime}\right)  ^{2}g^{2}\right)  }%
\end{array}
\right\}  ,\\
\\
\left.
\begin{array}
[c]{l}%
K_{\overset{psd-n}{C_{3}}}=\frac{-r\left(  1-{\small r}^{\prime2}\right)
k_{1}^{2}\sinh^{2}f{\small +4}r^{\prime\prime}\left(  1-r^{\prime2}%
-rr^{\prime\prime}\right)  {\small g}^{2}-2\sqrt{1-{\small r}^{\prime2}%
}\left(  1-r^{\prime2}-2rr^{\prime\prime}\right)  k_{1}g\sinh f}{r^{2}\left(
r\sqrt{1-{\small r}^{\prime2}}k_{1}\sinh f+2\left(  1-r^{\prime2}%
-rr^{\prime\prime}\right)  g\right)  ^{2}}\\
H_{\overset{psd-n}{C_{3}}}=\frac{-2r\left(  1-{\small r}^{\prime2}\right)
^{3/2}k_{1}g\sinh f+3r^{2}\left(  1-{\small r}^{\prime2}\right)  k_{1}%
^{2}\sinh^{2}f-4\left(  1-r^{\prime2}-rr^{\prime\prime}\right)  \left(
2-2r^{\prime2}-3rr^{\prime\prime}\right)  g^{2}}{3\left(  -r^{3}\left(
1-{\small r}^{\prime2}\right)  k_{1}^{2}\sinh^{2}f+4r\left(  1-r^{\prime
2}-rr^{\prime\prime}\right)  ^{2}g^{2}\right)  }%
\end{array}
\right\}  ,\\
\\
\left.
\begin{array}
[c]{l}%
K_{\overset{psd-n}{C_{4}}}=\frac{-r\left(  {\small r}^{\prime2}-1\right)
k_{1}^{2}\cosh^{2}f{\small -4}r^{\prime\prime}\left(  -1+r^{\prime
2}+rr^{\prime\prime}\right)  {\small g}^{2}+2\sqrt{-1+{\small r}^{\prime2}%
}\left(  -1+r^{\prime2}+2rr^{\prime\prime}\right)  k_{1}g\cosh f}{r^{2}\left(
r\sqrt{{\small r}^{\prime2}-1}k_{1}\cosh f-2\left(  -1+r^{\prime2}%
+rr^{\prime\prime}\right)  g\right)  ^{2}}\\
H_{\overset{psd-n}{C_{4}}}=\frac{-2r\left(  {\small r}^{\prime2}-1\right)
^{3/2}k_{1}g\cosh f-3r^{2}\left(  {\small r}^{\prime2}-1\right)  k_{1}%
^{2}\cosh^{2}f+4\left(  -1+r^{\prime2}+rr^{\prime\prime}\right)  \left(
-2+2r^{\prime2}+3rr^{\prime\prime}\right)  g^{2}}{3\left(  r^{3}\left(
{\small r}^{\prime2}-1\right)  k_{1}^{2}\cosh^{2}f-4r\left(  -1+r^{\prime
2}+rr^{\prime\prime}\right)  ^{2}g^{2}\right)  }%
\end{array}
\right\}  ,\\
\\
\left.
\begin{array}
[c]{l}%
K_{\overset{psd-n}{C_{5}}}=\frac{r\left(  1+{\small r}^{\prime2}\right)
k_{1}^{2}\cosh^{2}f+{\small 4r^{\prime\prime}\left(  1+r^{\prime2}%
+rr^{\prime\prime}\right)  g}^{2}+2\sqrt{1+{\small r}^{\prime2}}\left(
1+r^{\prime2}+2rr^{\prime\prime}\right)  k_{1}g\cosh f}{r^{2}\left(
r\sqrt{1+{\small r}^{\prime2}}k_{1}\cosh f+2\left(  1+r^{\prime2}%
+rr^{\prime\prime}\right)  g\right)  ^{2}}\\
H_{\overset{psd-n}{C_{5}}}=\frac{-2r\left(  1+{\small r}^{\prime2}\right)
^{3/2}k_{1}g\cosh f+3r^{2}\left(  1+{\small r}^{\prime2}\right)  k_{1}%
^{2}\cosh^{2}f-4\left(  1+r^{\prime2}+rr^{\prime\prime}\right)  \left(
2+2r^{\prime2}+3rr^{\prime\prime}\right)  g^{2}}{3\left(  r^{3}\left(
1+{\small r}^{\prime2}\right)  k_{1}^{2}\cosh^{2}f-4r\left(  1+r^{\prime
2}+rr^{\prime\prime}\right)  ^{2}g^{2}\right)  }%
\end{array}
\right\}  .
\end{array}
\right\}  \label{KHpsdn}%
\end{equation}

\end{theorem}

\begin{proof}
Here we will obtain the unit normal vector field, Gaussian and mean curvatures
of the canal hypersurfaces $\overset{psd-n}{C_{1}}(s,t,w)$ given by%
\begin{equation}%
\begin{array}
[c]{l}%
\overset{psd-n}{C_{1}}(s,t,w)=\gamma(s)-r(s)r^{\prime}(s)F_{1}(s)\\
\text{ \ \ }+r(s)\sqrt{1-r^{\prime}(s)^{2}}\left(  g(t,w)\sin(f(t,w))F_{2}%
(s)+\cos(f(t,w))F_{3}(s)+\frac{\sin(f(t,w))}{2g(t,w)}F_{4}(s)\right)  ,
\end{array}
\label{X1}%
\end{equation}
in $E_{1}^{4}$. Firstly, from (\ref{PsnullFrenet}), the first derivatives of
the canal hypersurface (\ref{X1}) are obtained as%
\begin{align}
&  (\overset{psd-n}{C_{1}})_{s}=A_{1}F_{1}+A_{2}F_{2}+A_{3}F_{3}+A_{4}%
F_{4},\label{9}\\
&  (\overset{psd-n}{C_{1}})_{t}=r\sqrt{1-r^{\prime2}}\left(  gf_{t}\cos
f+g_{t}\sin f\right)  F_{2}-r\sqrt{1-r^{\prime2}}f_{t}\sin fF_{3}\label{10}\\
&  \text{ \ \ \ \ \ \ \ \ \ \ \ }+\frac{r\sqrt{1-r^{\prime2}}}{2g^{2}}\left(
gf_{t}\cos f-g_{t}\sin f\right)  F_{4},\nonumber\\
&  (\overset{psd-n}{C_{1}})_{w}=r\sqrt{1-r^{\prime2}}\left(  gf_{w}\cos
f+g_{w}\sin f\right)  F_{2}-r\sqrt{1-r^{\prime2}}f_{w}\sin fF_{3}\label{11}\\
&  \text{ \ \ \ \ \ \ \ \ \ \ \ \ }+\frac{r\sqrt{1-r^{\prime2}}}{2g^{2}%
}\left(  gf_{w}\cos f-g_{w}\sin f\right)  F_{4},\nonumber
\end{align}
where%
\[
\left.
\begin{array}
[c]{l}%
A_{1}=1-r^{\prime}{}^{2}-rr^{\prime\prime}-\frac{r\sqrt{1-r^{\prime}{}^{2}%
}k_{1}\sin f}{2g},\\
A_{2}=\frac{1}{\sqrt{1-r^{\prime2}}}\left(  r\left(  1-r^{\prime}{}%
^{2}\right)  k_{3}\cos f-r^{\prime}\left(  r\sqrt{1-r^{\prime2}}k_{1}-\left(
1-r^{\prime}{}^{2}-rr^{\prime\prime}\right)  g\sin f\right)  \right)  ,\\
A_{3}=\frac{\left(  r\left(  1-r^{\prime}{}^{2}\right)  \left(  2k_{2}%
g^{2}-k_{3}\right)  \sin f\right)  +2r^{\prime}\left(  1-r^{\prime}{}%
^{2}-rr^{\prime\prime}\right)  g\cos f}{2g\sqrt{1-r^{\prime}{}^{2}}},\\
A_{4}=\frac{-2r\left(  1-r^{\prime}{}^{2}\right)  k_{2}g\cos f+r^{\prime
}\left(  1-r^{\prime}{}^{2}-rr^{\prime\prime}\right)  \sin f}{2g\sqrt
{1-r^{\prime}{}^{2}}}.
\end{array}
\right\}
\]
From (\ref{4y}) and (\ref{9})-(\ref{11}), the unit normal vector field of
$\overset{psd-n}{C_{1}}$ in $E^{4}$ is%
\begin{equation}
N=-r^{\prime}F_{1}+\sqrt{1-r^{\prime2}}\left(  g\sin fF_{2}+\cos fF_{3}%
+\frac{\sin f}{2g}F_{4}\right)  . \label{12}%
\end{equation}
and we get $\left\langle N,N\right\rangle =1$. Also, the coefficients of the
first fundamental form are given by%
\begin{equation}
\left.
\begin{array}
[c]{l}%
g_{11}=\frac{1}{4g^{2}\left(  1-r^{\prime2}\right)  }\left(
\begin{array}
[c]{l}%
\left(  1-r^{\prime2}\right)  \left(  r\sqrt{1-r^{\prime2}}k_{1}\sin
f-2\left(  1-r^{\prime2}-rr^{\prime\prime}\right)  g\right)  ^{2}\\
+\left(  r\left(  1-r^{\prime2}\right)  \left(  2k_{2}g^{2}-k_{3}\right)  \sin
f+2r^{\prime}\left(  1-r^{\prime2}-rr^{\prime\prime}\right)  g\cos f\right)
^{2}\\
+4g\left(
\begin{array}
[c]{l}%
2r\left(  1-r^{\prime2}\right)  k_{2}g\cos f\\
-r^{\prime}\left(  1-r^{\prime2}-rr^{\prime\prime}\right)  \sin f
\end{array}
\right)  \left(
\begin{array}
[c]{l}%
-r\left(  1-r^{\prime2}\right)  k_{3}\cos f\\
+r^{\prime}\left(
\begin{array}
[c]{l}%
r\sqrt{1-r^{\prime2}}k_{1}\\
-\left(  1-r^{\prime2}-rr^{\prime\prime}\right)  g\sin f
\end{array}
\right)
\end{array}
\right)
\end{array}
\right)  ,\\
\\
g_{12}=g_{21}=\frac{r^{2}}{2g^{2}}\left(
\begin{array}
[c]{l}%
-2\left(  1-r^{\prime2}\right)  k_{2}g^{3}f_{t}-\left(  r^{\prime}%
\sqrt{1-r^{\prime2}}k_{1}\cos f-\left(  1-r^{\prime2}\right)  k_{3}\right)
gf_{t}\\
-\left(  1-r^{\prime2}\right)  k_{2}g^{2}g_{t}\sin\left(  2f\right)  +\left(
r^{\prime}\sqrt{1-r^{\prime2}}k_{1}-\left(  1-r^{\prime2}\right)  k_{3}\cos
f\right)  g_{t}\sin f
\end{array}
\right)  ,\\
\\
g_{13}=g_{31}=\frac{r^{2}}{2g^{2}}\left(
\begin{array}
[c]{l}%
-2\left(  1-r^{\prime2}\right)  k_{2}g^{3}f_{w}-\left(  r^{\prime}%
\sqrt{1-r^{\prime2}}k_{1}\cos f-\left(  1-r^{\prime2}\right)  k_{3}\right)
gf_{w}\\
-\left(  1-r^{\prime2}\right)  k_{2}g^{2}g_{w}\sin\left(  2f\right)  +\left(
r^{\prime}\sqrt{1-r^{\prime2}}k_{1}-\left(  1-r^{\prime2}\right)  k_{3}\cos
f\right)  g_{w}\sin f
\end{array}
\right)  ,\\
\\
g_{22}=\frac{r^{2}}{g^{2}}\left(  1-r^{\prime2}\right)  \left(  g^{2}f_{t}%
^{2}-g_{t}^{2}\sin^{2}f\right)  ,\\
\\
g_{23}=g_{32}=\frac{r^{2}}{g^{2}}\left(  1-r^{\prime2}\right)  \left(
g^{2}f_{t}f_{w}-g_{t}g_{w}\sin^{2}f\right)  ,\\
\\
g_{33}=\frac{r^{2}}{g^{2}}\left(  1-r^{\prime2}\right)  \left(  g^{2}f_{w}%
^{2}-g_{w}^{2}\sin^{2}f\right)  ,
\end{array}
\right\}  \label{13}%
\end{equation}
and it follows that%
\begin{equation}
\det[g_{ij}]=\frac{-r^{4}\left(  1-r^{\prime2}\right)  \left(  r\sqrt
{1-r^{\prime2}}k_{1}\sin f-2\left(  1-r^{\prime2}-rr^{\prime\prime}\right)
g\right)  ^{2}\left(  g_{w}f_{t}-f_{w}g_{t}\right)  ^{2}\sin^{2}f}{4g^{4}}.
\label{14}%
\end{equation}

Now, for obtaining the coefficients of the second fundamental form, let us
give the second derivatives $(\overset{psd-n}{C})_{x_{i}x_{j}}=\frac
{\partial^{2}\overset{psd-n}{C}}{\partial x_{i}x_{j}}$ of the canal
hypersurface (\ref{X1}):%
\begin{equation}
(\overset{psd-n}{C_{1}})_{ss}=B_{1}F_{1}+B_{2}F_{2}+B_{3}F_{3}+B_{4}F_{4},
\label{15}%
\end{equation}%
\begin{align}
&  (\overset{psd-n}{C_{1}})_{st}=(\overset{psd-n}{C_{1}})_{ts}=\frac
{r\sqrt{1-r^{\prime2}}k_{1}\left(  -gf_{t}\cos f+g_{t}\sin f\right)  }{2g^{2}%
}F_{1}\label{16'}\\
&  \text{ \ \ \ }-\frac{r\left(  1-r^{\prime2}\right)  k_{3}f_{t}\sin
f-r^{\prime}\left(  1-r^{\prime2}-rr^{\prime\prime}\right)  \left(  gf_{t}\cos
f+g_{t}\sin f\right)  }{\sqrt{1-r^{\prime2}}}F_{2}\nonumber\\
&  \text{ \ \ \ }+\frac{\left(
\begin{array}
[c]{l}%
\left(  r\left(  1-r^{\prime2}\right)  \left(  2k_{2}g^{2}-k_{3}\right)  \cos
f-2r^{\prime}\left(  1-r^{\prime2}-rr^{\prime\prime}\right)  g\sin f\right)
gf_{t}\\
+r\left(  1-r^{\prime2}\right)  \left(  2g^{2}k_{2}+k_{3}\right)  g_{t}\sin f
\end{array}
\right)  }{2\sqrt{1-r^{\prime2}}g^{2}}F_{3}\nonumber\\
&  \text{ \ \ \ }+\frac{\left(  \left(  2r\left(  1-r^{\prime2}\right)
k_{2}g\sin f+r^{\prime}\left(  1-r^{\prime2}-rr^{\prime\prime}\right)  \cos
f\right)  gf_{t}-r^{\prime}\left(  1-r^{\prime2}-rr^{\prime\prime}\right)
g_{t}\sin f\right)  }{2\sqrt{1-r^{\prime2}}g^{2}}F_{4}\nonumber
\end{align}%
\begin{align}
&  (\overset{psd-n}{C_{1}})_{sw}=(\overset{psd-n}{C_{1}})_{ws}=\frac
{r\sqrt{1-r^{\prime2}}k_{1}}{2g^{2}}\left(  -gf_{w}\cos f+g_{w}\sin f\right)
F_{1}\label{17}\\
&  \text{ \ \ \ \ }+\frac{1}{\sqrt{1-r^{\prime2}}}\left(  -r\left(
1-r^{\prime2}\right)  k_{3}f_{w}\sin f+r^{\prime}\left(  1-r^{\prime
2}-rr^{\prime\prime}\right)  \left(  gf_{w}\cos f+g_{w}\sin f\right)  \right)
F_{2}\nonumber\\
&  \text{ \ \ \ \ }+\frac{1}{2\sqrt{1-r^{\prime2}}g^{2}}\left(
\begin{array}
[c]{l}%
\left(  r\left(  1-r^{\prime2}\right)  \left(  2k_{2}g^{2}-k_{3}\right)  \cos
f-2r^{\prime}\left(  1-r^{\prime2}-rr^{\prime\prime}\right)  g\sin f\right)
gf_{w}\\
+r\left(  1-r^{\prime2}\right)  \left(  2k_{2}g^{2}+k_{3}\right)  g_{w}\sin f
\end{array}
\right)  F_{3}\nonumber\\
&  \text{ \ \ \ \ }+\frac{1}{2\sqrt{1-r^{\prime2}}g^{2}}\left(
\begin{array}
[c]{l}%
\left(  2r\left(  1-r^{\prime2}\right)  k_{2}g\sin f+r^{\prime}\left(
1-r^{\prime2}-rr^{\prime\prime}\right)  \cos f\right)  gf_{w}\\
-r^{\prime}\left(  1-r^{\prime2}-rr^{\prime\prime}\right)  g_{w}\sin f
\end{array}
\right)  F_{4}\nonumber
\end{align}%
\begin{align}
&  (\overset{psd-n}{C_{1}})_{tt}=r\sqrt{1-r^{\prime2}}\left(  \left(
2f_{t}g_{t}+gf_{tt}\right)  \cos f+\left(  -gf_{t}^{2}+g_{tt}\right)  \sin
f\right)  F_{2}\label{18}\\
&  \text{ \ \ \ }-r\sqrt{1-r^{\prime2}}\left(  f_{t}^{2}\cos f+f_{tt}\sin
f\right)  F_{3}\nonumber\\
&  \text{ \ \ \ }-\frac{r\sqrt{1-r^{\prime2}}}{2g^{3}}\left(  g\left(
2f_{t}g_{t}-gf_{tt}\right)  \cos f+\left(  -2g_{t}^{2}+g\left(  gf_{t}%
^{2}+g_{tt}\right)  \right)  \sin f\right)  F_{4}\nonumber
\end{align}%
\begin{align}
&  (\overset{psd-n}{C_{1}})_{tw}=(\overset{psd-n}{C_{1}})_{wt}=r\sqrt
{1-r^{\prime2}}\left(  \left(  f_{t}g_{w}+f_{w}g_{t}+gf_{tw}\right)  \cos
f+\left(  -gf_{t}f_{w}+g_{tw}\right)  \sin f\right)  F_{2}\label{19}\\
&  \text{ \ \ \ }-r\sqrt{1-r^{\prime2}}\left(  f_{t}f_{w}\cos f+f_{tw}\sin
f\right)  F_{3}\nonumber\\
&  \text{ \ \ \ }-\frac{r\sqrt{1-r^{\prime2}}}{2g^{3}}\left(  g\left(
f_{t}g_{w}+f_{w}g_{t}-gf_{tw}\right)  \cos f+\left(  -2g_{t}g_{w}+g\left(
gf_{t}f_{w}+g_{tw}\right)  \right)  \sin f\right)  F_{4}\nonumber
\end{align}
and
\begin{align}
&  (\overset{psd-n}{C_{1}})_{ww}=r\sqrt{1-r^{\prime2}}\left(  \left(
2f_{w}g_{w}+gf_{ww}\right)  \cos f+\left(  -gf_{w}^{2}+g_{ww}\right)  \sin
f\right)  F_{2}\label{20}\\
&  \text{ \ \ \ }-r\sqrt{1-r^{\prime2}}\left(  f_{w}^{2}\cos f+f_{ww}\sin
f\right)  F_{3}\nonumber\\
&  \text{ \ \ \ }-\frac{r\sqrt{1-r^{\prime2}}}{2g^{3}}\left(  g\left(
2f_{w}g_{w}-gf_{ww}\right)  \cos f+\left(  -2g_{w}^{2}+g\left(  gf_{w}%
^{2}+g_{ww}\right)  \right)  \sin f\right)  F_{4},\nonumber
\end{align}
where%
\[%
\begin{array}
[c]{l}%
B_{1}=\frac{-r\sqrt{1-r^{\prime2}}k_{1}^{\prime}\sin f}{2g}+\frac{k_{1}\left(
r\left(  1-r^{\prime2}\right)  k_{2}g\cos f-r^{\prime}\left(  1-r^{\prime
2}-rr^{\prime\prime}\right)  \sin f\right)  }{g\sqrt{1-r^{\prime2}}%
}-3r^{\prime}r^{\prime\prime}-rr^{\prime\prime\prime},\\
\\
B_{2}=\frac{-r\sqrt{1-r^{\prime2}}k_{1}^{2}\sin f}{2g}+\left(  1-2r^{\prime
2}-2rr^{\prime\prime}\right)  k_{1}\\
-\frac{1}{2g\left(  1-r^{\prime2}\right)  ^{3/2}}\left(
\begin{array}
[c]{l}%
r\left(  1-r^{\prime2}\right)  ^{2}k_{3}^{2}\sin f+2\left(  1-r^{\prime
2}\right)  g\left(
\begin{array}
[c]{l}%
rr^{\prime}\sqrt{1-r^{\prime2}}k_{1}^{\prime}\\
-\left(
\begin{array}
[c]{l}%
\left(  rk_{3}^{\prime}+2r^{\prime}k_{3}\right)  \left(  1-r^{\prime2}\right)
\\
-2rr^{\prime}r^{\prime\prime}k_{3}%
\end{array}
\right)  \cos f
\end{array}
\right) \\
-2g^{2}\left(  r\left(  1-r^{\prime2}\right)  ^{2}k_{2}k_{3}+r^{\prime\prime
}\left(  1-4r^{\prime2}+3r^{\prime4}-rr^{\prime\prime}\right)  -rr^{\prime
}r^{\prime\prime\prime}\left(  1-r^{\prime2}\right)  \right)  \sin f
\end{array}
\right)  ,\\
\\
B_{3}=-\frac{1}{2g\left(  1-r^{\prime2}\right)  ^{3/2}}\left(
\begin{array}
[c]{l}%
2\left(  1-r^{\prime2}\right)  \left(  -\left(  1-r^{\prime2}\right)  \left(
rk_{2}^{\prime}+2r^{\prime}k_{2}\right)  +2rr^{\prime}r^{\prime\prime}%
k_{2}\right)  g^{2}\sin f\\
-\left(  1-r^{\prime2}\right)  \left(  -\left(  1-r^{\prime2}\right)  \left(
rk_{3}^{\prime}+2r^{\prime}k_{3}\right)  +2rr^{\prime}r^{\prime\prime}%
k_{3}\right)  \sin f\\
-2g\left(
\begin{array}
[c]{l}%
-r\left(  1-r^{\prime2}\right)  k_{2}\left(  r^{\prime}\sqrt{1-r^{\prime2}%
}k_{1}-2\left(  1-r^{\prime2}\right)  k_{3}\cos f\right) \\
+\left(  r^{\prime\prime}\left(  1-4r^{\prime2}+3r^{\prime4}-rr^{\prime\prime
}\right)  -rr^{\prime}r^{\prime\prime\prime}\left(  1-r^{\prime2}\right)
\right)  \cos f
\end{array}
\right)
\end{array}
\right)  ,\\
\\
B_{4}=-\frac{1}{2g\left(  1-r^{\prime2}\right)  ^{3/2}}\left(
\begin{array}
[c]{l}%
-2\left(  1-r^{\prime2}\right)  \left(  -\left(  1-r^{\prime2}\right)  \left(
rk_{2}^{\prime}+2r^{\prime}k_{2}\right)  +2rr^{\prime}r^{\prime\prime}%
k_{2}\right)  g\cos f\\
+\left(
\begin{array}
[c]{l}%
2r\left(  1-r^{\prime2}\right)  ^{2}k_{2}^{2}g^{2}-r\left(  1-r^{\prime
2}\right)  ^{2}k_{2}k_{3}\\
-r^{\prime\prime}\left(  1-4r^{\prime2}+3r^{\prime4}-rr^{\prime\prime}\right)
+rr^{\prime}r^{\prime\prime\prime}\left(  1-r^{\prime2}\right)
\end{array}
\right)  \sin f
\end{array}
\right)  .
\end{array}
\]
Thus, from (\ref{4y}), (\ref{12}) and (\ref{15})-(\ref{20}), the coefficients
of the second fundamental form are given by%
\begin{equation}
\left.
\begin{array}
[c]{l}%
h_{11}=\frac{-1}{4g^{2}\left(  1-r^{\prime2}\right)  }\left(
\begin{array}
[c]{l}%
r\left(  1-r^{\prime2}\right)  ^{2}k_{1}^{2}\sin^{2}f+4r\left(  1-r^{\prime
2}\right)  ^{2}k_{2}^{2}g^{4}\sin^{2}f+r\left(  1-r^{\prime2}\right)
^{2}k_{3}^{2}\sin^{2}f\\
-2g^{2}\left(  r\left(  1-r^{\prime2}\right)  ^{2}k_{2}k_{3}\left(
3+\cos\left(  2f\right)  \right)  +2r^{\prime\prime}\left(  1-r^{\prime
2}-rr^{\prime\prime}\right)  \right) \\
+2\sqrt{1-r^{\prime2}}k_{1}g\left(  4rr^{\prime}\left(  1-r^{\prime2}\right)
k_{2}g\cos f-\left(  1-r^{\prime2}-2rr^{\prime\prime}\right)  \sin f\right)
\end{array}
\right)  ,\\
\\
h_{12}=h_{21}=\frac{-r}{2g^{2}}\left(
\begin{array}
[c]{l}%
-2\left(  1-r^{\prime2}\right)  k_{2}g^{3}f_{t}-\left(  r^{\prime}%
\sqrt{1-r^{\prime2}}k_{1}\cos f-\left(  1-r^{\prime2}\right)  k_{3}\right)
gf_{t}\\
-\left(  1-r^{\prime2}\right)  k_{2}g^{2}g_{t}\sin\left(  2f\right)  +\left(
r^{\prime}\sqrt{1-r^{\prime2}}k_{1}-\left(  1-r^{\prime2}\right)  k_{3}\cos
f\right)  g_{t}\sin f
\end{array}
\right)  ,\\
\\
h_{13}=h_{31}=\frac{-r}{2g^{2}}\left(
\begin{array}
[c]{l}%
-2\left(  1-r^{\prime2}\right)  k_{2}g^{3}f_{w}-\left(  r^{\prime}%
\sqrt{1-r^{\prime2}}k_{1}\cos f-\left(  1-r^{\prime2}\right)  k_{3}\right)
gf_{w}\\
-\left(  1-r^{\prime2}\right)  k_{2}g^{2}g_{w}\sin\left(  2f\right)  +\left(
r^{\prime}\sqrt{1-r^{\prime2}}k_{1}-\left(  1-r^{\prime2}\right)  k_{3}\cos
f\right)  g_{w}\sin f
\end{array}
\right)  ,\\
\\
h_{22}=\frac{r\left(  1-r^{\prime2}\right)  \left(  g_{t}^{2}\sin^{2}%
f-g^{2}f_{t}^{2}\right)  }{g^{2}},\\
\\
h_{23}=h_{32}=\frac{r\left(  1-r^{\prime2}\right)  \left(  g_{t}g_{w}\sin
^{2}f-g^{2}f_{t}f_{w}\right)  }{g^{2}},\\
\\
h_{33}=\frac{r\left(  1-r^{\prime2}\right)  \left(  g_{w}^{2}\sin^{2}%
f-g^{2}f_{w}^{2}\right)  }{g^{2}}%
\end{array}
\right\}  \label{21}%
\end{equation}
and it implies%
\begin{equation}
\det[h_{ij}]=\frac{r^{2}\left(  1-r^{\prime2}\right)  }{4g^{4}}\left(
\begin{array}
[c]{l}%
r\left(  1-r^{\prime2}\right)  k_{1}^{2}\sin^{2}f-4r^{\prime\prime}\left(
1-r^{\prime2}-rr^{\prime\prime}\right)  g^{2}\\
-2\sqrt{1-r^{\prime2}}\left(  1-r^{\prime2}-2rr^{\prime\prime}\right)
k_{1}g\sin f
\end{array}
\right)  \left(  f_{t}g_{w}-f_{w}g_{t}\right)  ^{2}\sin^{2}f. \label{22}%
\end{equation}

So, from (\ref{yy4}), (\ref{14}) and (\ref{22}), we obtain the Gaussian
curvature $K_{\overset{psd-n}{C_{1}}}$ as given in (\ref{KHpsdn}).

Also, using (\ref{13}) and (\ref{21}) in (\ref{7yyy}), the shape operator of
the canal hypersurface (\ref{X1}) can be obtained and so using the shape
operator, we get the mean curvature $H_{\overset{psd-n}{C_{1}}}$ as given in
(\ref{KHpsdn}).

Similarly, the Gaussian and mean curvatures $K_{\overset{psd-n}{C_{i}}}$ and
$H_{\overset{psd-n}{C_{i}}}$ of the canal hypersurfaces $\overset{psd-n}{C_{i}%
},i\in\{2,3,4,5\},$ given by (\ref{KHpsdn}) can be obtained.
\end{proof}

Here, from (\ref{KHpsdn}), we can state the following theorem which gives an
important relation between Gaussian and mean curvatures of the canal hypersurfaces:

\begin{proposition}
The Gaussian curvatures and the mean curvatures of the canal hypersurfaces
$\overset{psd-n}{C_{1}},$ $\overset{psd-n}{C_{3}}$ and $\overset{psd-n}{C_{4}%
}$ given by (\ref{canal1})\ in $E_{1}^{4}$ satisfy%
\begin{equation}
3H_{\overset{psd-n}{C_{i}}}-r^{2}K_{\overset{psd-n}{C_{i}}}+\frac{2}%
{r}=0,\text{ }i=1,3,4\label{y2y}%
\end{equation}
and the Gaussian curvature and the mean curvature of the canal hypersurfaces
$\overset{psd-n}{C_{2}}$ and $\overset{psd-n}{C_{5}}$ given by (\ref{canal1})
and (\ref{canal2})\ in $E_{1}^{4}$ satisfy%
\begin{equation}
3H_{\overset{psd-n}{C_{i}}}-r^{2}K_{\overset{psd-n}{C_{i}}}-\frac{2}%
{r}=0,\text{ }i=2,5.\label{y3y}%
\end{equation}

\end{proposition}

Now, we will give some results for the canal hypersurface
$\overset{psd-n}{C_{1}}$ given by (\ref{canal1})\ in $E_{1}^{4}.$ Similarly,
one can obtain similar results for the canal hypersurfaces
$\overset{psd-n}{C_{2}},$ $\overset{psd-n}{C_{3}},$ $\overset{psd-n}{C_{4}}$
and $\overset{psd-n}{C_{5}}$ given by (\ref{canal1}) and (\ref{canal2})\ in
$E_{1}^{4},$ too.

\begin{proposition}
Let the pseudo null curve $\gamma(s),$ which generates the canal hypersurface
$\overset{psd-n}{C_{1}}$ given by (\ref{canal1})\ in $E_{1}^{4},$ be a
straight line. Then, the canal hypersurface $\overset{psd-n}{C_{1}}$ is flat,
if $r(s)=as+b,$ $a,b\in%
\mathbb{R}
,$ $a\neq\pm1.$
\end{proposition}

\begin{proof}
If the pseudo null curve $\gamma(s)$ is a straight line, then we have
$k_{1}(s)=0.$ So from (\ref{KHpsdn}), we have%
\begin{equation}
K_{\overset{psd-n}{C_{1}}}=\frac{r^{\prime\prime}}{r^{2}(1-r^{\prime
2}-rr^{\prime\prime})}. \label{kt1}%
\end{equation}
If we use $r(s)=as+b,$ $a,b\in%
\mathbb{R}
,$ $a\neq\pm1$ in (\ref{kt1}), then $K_{\overset{psd-n}{C_{1}}}$ vanishes and
this completes the proof.
\end{proof}

\begin{proposition}
Let the pseudo null curve $\gamma(s),$ which generates the canal hypersurface
$\overset{psd-n}{C_{1}}$ given by (\ref{canal1})\ in $E_{1}^{4},$ not be a
straight line. If $g(t,w)\neq\sin(f(t,w))$, then the canal hypersurface
$\overset{psd-n}{C_{1}}$ cannot be flat.

Also, if $g(t,w)=\sin(f(t,w))$, then the canal hypersurface
$\overset{psd-n}{C_{1}}$ is flat when the equation%
\begin{equation}
2\left(  1-r^{\prime2}\right)  \left(  \sqrt{1-r^{\prime2}}+2r^{\prime\prime
}\right)  -r\left(  1-r^{\prime2}+4r^{\prime\prime}\left(  \sqrt{1-r^{\prime
2}}+r^{\prime\prime}\right)  \right)  =0 \label{kt5}%
\end{equation}
holds.
\end{proposition}

\begin{proof}
If the pseudo null curve $\gamma(s),$ which generates the canal hypersurface
$\overset{psd-n}{C_{1}}$ given by (\ref{canal1})\ in $E_{1}^{4},$ isn't a
straight line, then we have $k_{1}(s)=1.$ So, from (\ref{KHpsdn}) we get%
\begin{equation}
K_{\overset{psd-n}{C_{1}}}=\frac{-r(1-r^{\prime2})\sin^{2}f+4r^{\prime\prime
}(1-r^{\prime2}-rr^{\prime\prime})g^{2}+2\sqrt{1-r^{\prime2}}(1-r^{\prime
2}-2rr^{\prime\prime})g\sin f}{r^{2}\left(  r\sqrt{1-{\small r}^{\prime2}}\sin
f-2\left(  1-r^{\prime2}-rr^{\prime\prime}\right)  g\right)  ^{2}}.
\label{kt2}%
\end{equation}
From (\ref{kt2}), if the canal hypersurface $\overset{psd-n}{C_{1}}$ is flat,
then we have%
\begin{equation}
-r(1-r^{\prime2})\sin^{2}f+4r^{\prime\prime}(1-r^{\prime2}-rr^{\prime\prime
})g^{2}+2\sqrt{1-r^{\prime2}}(1-r^{\prime2}-2rr^{\prime\prime})g\sin f=0.
\label{kt3}%
\end{equation}
Firstly, let us suppose $g(t,w)\neq\sin(f(t,w))$. Since the set $\{\sin
^{2}f,g^{2},g\sin f\}$ is linearly independent, we get from (\ref{kt3}) that%
\begin{equation}
-r(1-r^{\prime2})=4r^{\prime\prime}(1-r^{\prime2}-rr^{\prime\prime}%
)=2\sqrt{1-r^{\prime2}}(1-r^{\prime2}-2rr^{\prime\prime})=0 \label{kt7}%
\end{equation}
and this is a contradiction. Thus, $\overset{psd-n}{C_{1}}$ cannot be flat in
this situation and this proves the first part of this Proposition.

Secondly, if $g(t,w)=\sin(f(t,w))$ holds, then from (\ref{kt2}) we have%
\begin{equation}
K_{\overset{psd-n}{C_{1}}}=\frac{2\left(  1-r^{\prime2}\right)  \left(
\sqrt{1-r^{\prime2}}+2r^{\prime\prime}\right)  -r\left(  1-r^{\prime
2}+4r^{\prime\prime}\left(  \sqrt{1-r^{\prime2}}+r^{\prime\prime}\right)
\right)  }{r^{2}\left(  2-2r^{\prime2}-r\left(  \sqrt{1-r^{\prime2}%
}+2r^{\prime\prime}\right)  \right)  ^{2}} \label{kt4}%
\end{equation}
and this completes the proof.
\end{proof}

\begin{proposition}
Let the pseudo null curve $\gamma(s),$ which generates the canal hypersurface
$\overset{psd-n}{C_{1}}$ given by (\ref{canal1})\ in $E_{1}^{4},$ be a
straight line. Then, the canal hypersurface $\overset{psd-n}{C_{1}}$ is
minimal, if $r(s)$ satisfies $%
{\displaystyle\int}
\frac{dr}{\sqrt{1-\left(  \frac{a}{r}\right)  ^{\frac{4}{3}}}}=\pm s+b,$
$a,b\in%
\mathbb{R}
$.
\end{proposition}

\begin{proof}
If the pseudo null curve $\gamma(s)$ is a straight line, from (\ref{KHpsdn}),
we have%
\begin{equation}
H_{\overset{psd-n}{C_{1}}}=-\frac{2-2r^{\prime2}-3rr^{\prime\prime}%
}{3r(1-r^{\prime2}-rr^{\prime\prime})}. \label{ht1}%
\end{equation}
So, if the equation%
\begin{equation}
2-2r^{\prime2}-3rr^{\prime\prime}=0 \label{ht2}%
\end{equation}
holds, then the canal hypersurface $\overset{psd-n}{C_{1}}$ is minimal. The
solution of the equation (\ref{ht2}) can be found in \cite{Altin} and this
completes the proof.
\end{proof}

\begin{proposition}
Let the pseudo null curve $\gamma(s),$ which generates the canal hypersurface
$\overset{psd-n}{C_{1}}$ given by (\ref{canal1})\ in $E_{1}^{4},$ not be a
straight line. If $g(t,w)\neq\sin(f(t,w))$, then the canal hypersurface
$\overset{psd-n}{C_{1}}$ cannot be minimal.

Also, if $g(t,w)=\sin(f(t,w))$, then the canal hypersurface
$\overset{psd-n}{C_{1}}$ is minimal when the equation%
\begin{equation}
8\left(  1-r^{\prime2}\right)  ^{2}-2r\left(  1-r^{\prime2}\right)  \left(
\sqrt{1-r^{\prime2}}+10r^{\prime\prime}\right)  -3r^{2}\left(  1-r^{\prime
2}-4r^{\prime\prime2}\right)  =0 \label{ht5}%
\end{equation}
holds.
\end{proposition}

\begin{proof}
If the pseudo null curve $\gamma(s),$ which generates the canal hypersurface
$\overset{psd-n}{C_{1}}$ given by (\ref{canal1})\ in $E_{1}^{4},$ isn't a
straight line, then from (\ref{KHpsdn}) we get%
\begin{equation}
H_{\overset{psd-n}{C_{1}}}=\frac{2r\left(  1-{\small r}^{\prime2}\right)
^{3/2}g\sin f+3r^{2}\left(  1-{\small r}^{\prime2}\right)  \sin^{2}f-4\left(
1-r^{\prime2}-rr^{\prime\prime}\right)  \left(  2-2r^{\prime2}-3rr^{\prime
\prime}\right)  g^{2}}{3\left(  -r^{3}\left(  1-{\small r}^{\prime2}\right)
\sin^{2}f+4r\left(  1-r^{\prime2}-rr^{\prime\prime}\right)  ^{2}g^{2}\right)
}. \label{ht4}%
\end{equation}
From (\ref{ht4}), if the canal hypersurface $\overset{psd-n}{C_{1}}$ is
minimal, then we have%
\begin{equation}
2r\left(  1-{\small r}^{\prime2}\right)  ^{3/2}g\sin f+3r^{2}\left(
1-{\small r}^{\prime2}\right)  \sin^{2}f-4\left(  1-r^{\prime2}-rr^{\prime
\prime}\right)  \left(  2-2r^{\prime2}-3rr^{\prime\prime}\right)  g^{2}=0.
\label{ht3}%
\end{equation}
Firstly, let us suppose that $g(t,w)\neq\sin(f(t,w)).$ Since the set $\{g\sin
f,\sin^{2}f,g^{2}\}$ is linearly independent, we get from (\ref{ht3}) that%
\begin{equation}
2r\left(  1-{\small r}^{\prime2}\right)  ^{3/2}=3r^{2}\left(  1-{\small r}%
^{\prime2}\right)  =4\left(  1-r^{\prime2}-rr^{\prime\prime}\right)  \left(
2-2r^{\prime2}-3rr^{\prime\prime}\right)  =0 \label{ht7}%
\end{equation}
and this is a contradiction. Thus, $\overset{psd-n}{C_{1}}$ cannot be minimal
in this situation and this proves the first part of this Proposition.

Secondly, if $g(t,w)=\sin(f(t,w))$ holds, then we have%
\begin{equation}
H_{\overset{psd-n}{C_{1}}}=-\frac{8\left(  1-r^{\prime2}\right)
^{2}-2r\left(  1-r^{\prime2}\right)  \left(  \sqrt{1-r^{\prime2}}%
+10r^{\prime\prime}\right)  -3r^{2}\left(  1-r^{\prime2}-4r^{\prime\prime
2}\right)  }{3r\left(  4\left(  1-r^{\prime2}\right)  ^{2}-8r\left(
1-r^{\prime2}\right)  r^{\prime\prime}-r^{2}\left(  1-r^{\prime2}%
-4r^{\prime\prime2}\right)  \right)  } \label{ht6}%
\end{equation}
and this completes the proof.
\end{proof}

Here, let us construct an example for the canal hypersurface
$\overset{psd-n}{C_{1}}(s,t,w)$ that is formed as the envelope of a family of
pseudo hyperspheres whose center lie on a pseudo null curve in $E_{1}^{4}.$
One can construct the canal hypersurfaces $\overset{psd-n}{C_{2}},$
$\overset{psd-n}{C_{3}},$ $\overset{psd-n}{C_{4}}$ and $\overset{psd-n}{C_{5}%
},$ similarly.

\begin{example}
Let us take the pseudo null curve (given in \cite{Kazim2})
\begin{equation}
{\small \gamma(s)=}\frac{1}{2\sqrt{2}}\left(  \cosh(2s),\sinh(2s),\sin
(2s),-\cos(2s)\right)  \label{ex1y}%
\end{equation}
in $E_{1}^{4}$. The Frenet vectors and curvatures of the curve (\ref{ex1y})
are%
\begin{equation}
\left.
\begin{array}
[c]{l}%
F_{1}=\frac{1}{\sqrt{2}}\left(  \sinh(2s),\cosh(2s),\cos(2s),\sin(2s)\right)
,\\
F_{2}=\sqrt{2}\left(  \cosh(2s),\sinh(2s),-\sin(2s),\cos(2s)\right)  ,\\
F_{3}=\frac{1}{\sqrt{2}}\left(  \sinh(2s),\cosh(2s),-\cos(2s),-\sin
(2s)\right)  ,\\
F_{4}=\frac{1}{2\sqrt{2}}\left(  -\cosh(2s),-\sinh(2s),-\sin(2s),\cos
(2s)\right)  ,\\
k_{1}=1,\text{ }k_{2}=4,\text{ }k_{3}=0.
\end{array}
\right\}  \label{ex2y}%
\end{equation}
If we assume that $g(t,w)=t$ and $f(t,w)=w$ in (\ref{canal1}), the canal
hypersurfaces $\overset{psd-n}{C_{1}}(s,t,w)$ can be parametrized by%
\begin{align}
&  \overset{psd-n}{C_{1}}(s,t,w)=\label{ex3y}\\
&  \frac{1}{16\sqrt{2}t}\left(
\begin{array}
[c]{l}%
\cosh(2s)\left(  8t+\sqrt{3}s\left(  -1+8t^{2}\right)  \sin w\right)
+4st\left(  -1+\sqrt{3}\cos w\right)  \sinh(2s),\\
\sinh(2s)\left(  8t+\sqrt{3}s\left(  -1+8t^{2}\right)  \sin w\right)
+4st\left(  -1+\sqrt{3}\cos w\right)  \cosh(2s),\\
\sin(2s)\left(  8t-\sqrt{3}s\left(  1+8t^{2}\right)  \sin w\right)
+4st\left(  -1-\sqrt{3}\cos w\right)  \cos(2s),\\
\cos(2s)\left(  -8t+\sqrt{3}s\left(  1+8t^{2}\right)  \sin w\right)
+4st\left(  -1-\sqrt{3}\cos w\right)  \sin(2s)
\end{array}
\right) \nonumber
\end{align}
where the radius function has been taken as $r(s)=\frac{s}{2}.$ From
(\ref{KHpsdn}), the Gaussian and mean curvatures of the canal hypersurfaces
$\overset{psd-n}{C_{1}}$ are obtained as%
\begin{equation}
\left.
\begin{array}
[c]{l}%
K_{\overset{psd-n}{C_{1}}}=-\frac{8}{s^{2}\left(  s-2\sqrt{3}t\csc w\right)
},\\
H_{\overset{psd-n}{C_{1}}}=\frac{-48t^{2}+2s\left(  2\sqrt{3}t+3s\sin
w\right)  \sin w}{36st^{2}-3s^{3}\sin^{2}w}.
\end{array}
\right\}  \label{ex5y}%
\end{equation}
In Figure 1 (a), one can see the projection of the canal hypersurfaces
(\ref{ex3y}) for $w=\frac{\pi}{3}$ into $x_{1}x_{3}x_{4}$-space.
\end{example}

\begin{theorem}
\textit{The canal hypersurfaces that are formed as the envelope of a family of
pseudo hyperspheres whose centers lie on a partially null curve }$\gamma(s)$
with Frenet vector fields $F_{i},$ $i\in\{1,2,3,4\},$\textit{ in }$E_{1}^{4}$
\textit{can be parametrized by}%
\begin{equation}
\left.
\begin{array}
[c]{l}%
\overset{part-n}{C_{1}}(s,t,w)=\gamma(s)-r(s)r^{\prime}(s)F_{1}(s)\\
\text{ \ \ }\mp r(s)\sqrt{1-r^{\prime}(s)^{2}}\left(  \sin(f(t,w))F_{2}%
(s)+g(t,w)\cos(f(t,w))F_{3}(s)+\frac{\cos(f(t,w))}{2g(t,w)}F_{4}(s)\right)
,\\
\overset{part-n}{C_{2}}(s,t,w)=\gamma(s)-r(s)r^{\prime}(s)F_{1}(s)\\
\text{ \ \ }\mp r(s)\sqrt{1-r^{\prime}(s)^{2}}\left(  \cos(f(t,w))F_{2}%
(s)+g(t,w)\sin(f(t,w))F_{3}(s)+\frac{\sin(f(t,w))}{2g(t,w)}F_{4}(s)\right)
,\\
\overset{part-n}{C_{3}}(s,t,w)=\gamma(s)-r(s)r^{\prime}(s)F_{1}(s)\\
\text{ \ \ }\mp r(s)\sqrt{1-r^{\prime}(s)^{2}}\left(  \cosh(f(t,w))F_{2}%
(s)+g(t,w)\sinh(f(t,w))F_{3}(s)-\frac{\sinh(f(t,w))}{2g(t,w)}F_{4}(s)\right)
,\\
\overset{part-n}{C_{4}}(s,t,w)=\gamma(s)-r(s)r^{\prime}(s)F_{1}(s)\\
\text{ \ \ }\mp r(s)\sqrt{r^{\prime}(s)^{2}-1}\left(  \sinh(f(t,w))F_{2}%
(s)+g(t,w)\cosh(f(t,w))F_{3}(s)-\frac{\cosh(f(t,w))}{2g(t,w)}F_{4}(s)\right)
,
\end{array}
\right\}  \label{prtn canal1}%
\end{equation}
where we suppose $r^{\prime}(s)^{2}<1$ for canal hypersurfaces
$\overset{part-n}{C_{1}},\overset{part-n}{C_{2}},\overset{part-n}{C_{3}}$ and
$r^{\prime}(s)^{2}>1$ for canal hypersurface $\overset{part-n}{C_{4}}$.

\textit{Furthermore, the canal hypersurfaces that are formed as the envelope
of a family of pseudo hyperbolic hyperspheres whose centers lie on a partially
null curve }$\gamma(s)$ with Frenet vector fields $F_{i},$ $i\in\{1,2,3,4\},$
\textit{in }$E_{1}^{4}$ \textit{can be parametrized by}%
\begin{align}
&  \overset{part-n}{C_{5}}(s,t,w)=\gamma(s)+r(s)r^{\prime}(s)F_{1}%
(s)\label{prtn canal2}\\
&  \text{ \ \ }\mp r(s)\sqrt{1+r^{\prime}(s)^{2}}\left(  \sinh(f(t,w))F_{2}%
(s)+g(t,w)\cosh(f(t,w))F_{3}(s)-\frac{\cosh(f(t,w))}{2g(t,w)}F_{4}(s)\right)
.\nonumber
\end{align}

Here, the canal hypersurfaces $\overset{part-n}{C_{1}},$
$\overset{part-n}{C_{2}},$ $\overset{part-n}{C_{3}},$ $\overset{part-n}{C_{4}%
}$ are timelike and the canal hypersurface $\overset{part-n}{C_{5}}$ is spacelike.
\end{theorem}

\begin{proof}
Let the center curve $\gamma:I\subseteq%
\mathbb{R}
\rightarrow E_{1}^{4}$ be a partially null curve with non-zero curvature with
Frenet vector fields. Then, the parametrization of the envelope of pseudo
hyperspheres (or pseudo hyperbolic hyperspheres) defining the canal
hypersurfaces $X(s,t,w)$ in $E_{1}^{4}$ can be given by%
\begin{equation}
\overset{part-n}{C}(s,t,w)-\gamma(s)=a_{1}(s,t,w)F_{1}(s)+a_{2}(s,t,w)F_{2}%
(s)+a_{3}(s,t,w)F_{3}(s)+a_{4}(s,t,w)F_{4}(s). \label{prtn 1}%
\end{equation}
Furthermore, since $C(s,t,w)$ lies on the pseudo hyperspheres ($\lambda=1$)
(or pseudo hyperbolic hyperspheres ($\lambda=-1$)), we have%
\begin{equation}
g(\overset{part-n}{C}(s,t,w)-\gamma(s),\overset{part-n}{C}(s,t,w)-\gamma
(s))=\lambda r^{2}(s) \label{prtn 2}%
\end{equation}
which leads to from (\ref{prtn 1}) and (\ref{PartnullFi ler}) that%
\begin{equation}
a_{1}^{2}+a_{2}^{2}+2a_{3}a_{4}=\lambda r^{2} \label{prtn 3}%
\end{equation}
and%
\begin{equation}
a_{1}a_{1_{s}}+a_{2}a_{2_{s}}+a_{4}a_{2_{s}}+a_{2}a_{4_{s}}=\lambda rr_{s}.
\label{prtn 4}%
\end{equation}

So, differentiating (\ref{prtn 1}) with respect to $s$ and using the Frenet
formula (\ref{PartnullFrenet}), we get%
\begin{align}
(\overset{part-n}{C})_{s}  &  =\left(  1-a_{2}k_{1}+a_{1_{s}}\right)
F_{1}+\left(  a_{1}k_{1}-a_{4}k_{2}+a_{2_{s}}\right)  F_{2}\label{prtn 5}\\
&  +\left(  a_{2}k_{2}+a_{3}k_{3}+a_{3_{s}}\right)  F_{3}+\left(  -a_{4}%
k_{3}+a_{4_{s}}\right)  F_{4}.\nonumber
\end{align}
Furthermore, $C(s,t,w)-\gamma(s)$ is a normal vector to the canal
hypersurfaces, which implies that%
\begin{equation}
g(\overset{part-n}{C}(s,t,w)-\gamma(s),(\overset{part-n}{C})_{s}(s,t,w))=0
\label{prtn 6}%
\end{equation}
and so, from (\ref{prtn 1}), (\ref{prtn 5}) and (\ref{prtn 6'}) we have%
\begin{equation}
\text{ }\left(
\begin{array}
[c]{l}%
a_{1}\left(  1-a_{2}k_{1}+a_{1_{s}}\right)  +a_{2}\left(  a_{1}k_{1}%
-a_{4}k_{2}+a_{2_{s}}\right) \\
+a_{3}\left(  -a_{4}k_{3}+a_{4_{s}}\right)  +a_{4}\left(  a_{2}k_{2}%
+a_{3}k_{3}+a_{3_{s}}\right)
\end{array}
\right)  =0. \label{prtn 6'}%
\end{equation}
Using (\ref{prtn 4}) in (\ref{prtn 6'}), we get
\begin{equation}
a_{1}=-\lambda rr_{s}. \label{prtn 7}%
\end{equation}
Hence, using (\ref{prtn 7}) in (\ref{prtn 3}), we reach that$a_{1}^{2}%
+a_{2}^{2}+2a_{3}a_{4}=\lambda r^{2}$%
\begin{equation}
{\small a}_{2}^{2}{\small +}2{\small a}_{3}a_{4}{\small =\lambda r}%
^{2}{\small (}1-\lambda{\small r}_{s}^{2}{\small ).} \label{prtn 8}%
\end{equation}
Therefore from (\ref{prtn 7}) and (\ref{prtn 8}), the canal hypersurfaces
$\overset{part-n}{C}(s,t,w)$ that are formed as the envelope of a family of
pseudo hyperspheres or pseudo hyperbolic hyperspheres\textit{ }whose centers
lie on a partially null curve in $E_{1}^{4}$ can be parametrized by
(\ref{prtn canal1}) or (\ref{prtn canal2}), respectively.
\end{proof}

Using the similar procedure with proof of the Theorem \ref{Teopseudonull}, we
can obtain the following Theorem:

\begin{theorem}
\textit{The Gaussian and mean curvatures of the canal hypersurfaces
}$\overset{part-n}{C_{i}},$ $i\in\{1,2,3,4\},$ given by (\ref{prtn canal1})
and (\ref{prtn canal2}) in $E_{1}^{4}$ are%
\begin{equation}
\left.
\begin{array}
[c]{l}%
\left.
\begin{array}
[c]{l}%
K_{\overset{part-n}{C_{1}}}=\frac{-r\left(  1-{\small r}^{\prime2}\right)
k_{1}^{2}\sin^{2}f{\small +}r^{\prime\prime}\left(  1-r^{\prime2}%
-rr^{\prime\prime}\right)  +\sqrt{1-{\small r}^{\prime2}}\left(  1-r^{\prime
2}-2rr^{\prime\prime}\right)  k_{1}\sin f}{r^{2}\left(  1-{\small r}^{\prime
2}-rr^{\prime\prime}-r\sqrt{1-{\small r}^{\prime2}}k_{1}\sin f\right)  ^{2}}\\
H_{\overset{part-n}{C_{1}}}=\frac{r\left(  1-{\small r}^{\prime2}\right)
^{3/2}k_{1}\sin f+3r^{2}\left(  1-{\small r}^{\prime2}\right)  k_{1}^{2}%
\sin^{2}f-\left(  1-r^{\prime2}-rr^{\prime\prime}\right)  \left(
2-2r^{\prime2}-3rr^{\prime\prime}\right)  }{3r\left(  -r^{2}\left(
1-{\small r}^{\prime2}\right)  k_{1}^{2}\sin^{2}f+\left(  1-r^{\prime
2}-rr^{\prime\prime}\right)  ^{2}\right)  }%
\end{array}
\right\}  ,\\
\\
\left.
\begin{array}
[c]{l}%
K_{\overset{part-n}{C_{2}}}=\frac{r\left(  1-{\small r}^{\prime2}\right)
k_{1}^{2}\cos^{2}f{\small -}r^{\prime\prime}\left(  1-r^{\prime2}%
-rr^{\prime\prime}\right)  -\sqrt{1-{\small r}^{\prime2}}\left(  1-r^{\prime
2}-2rr^{\prime\prime}\right)  k_{1}\cos f}{r^{2}\left(  1-r^{\prime
2}-rr^{\prime\prime}-r\sqrt{1-{\small r}^{\prime2}}k_{1}\cos f\right)  ^{2}}\\
H_{\overset{part-n}{C_{2}}}=\frac{-r\left(  1-{\small r}^{\prime2}\right)
^{3/2}k_{1}\cos f-3r^{2}\left(  1-{\small r}^{\prime2}\right)  k_{1}^{2}%
\cos^{2}f+\left(  1-r^{\prime2}-rr^{\prime\prime}\right)  \left(
2-2r^{\prime2}-3rr^{\prime\prime}\right)  }{3r\left(  -r^{2}\left(
1-{\small r}^{\prime2}\right)  k_{1}^{2}\cos^{2}f+\left(  1-r^{\prime
2}-rr^{\prime\prime}\right)  ^{2}\right)  }%
\end{array}
\right\}  ,\\
\\
\left.
\begin{array}
[c]{l}%
K_{\overset{part-n}{C_{3}}}=\frac{-r\left(  1-{\small r}^{\prime2}\right)
k_{1}^{2}\cosh^{2}f{\small +}r^{\prime\prime}\left(  1-r^{\prime2}%
-rr^{\prime\prime}\right)  +\sqrt{1-{\small r}^{\prime2}}\left(  1-r^{\prime
2}-2rr^{\prime\prime}\right)  k_{1}\cosh f}{r^{2}\left(  1-r^{\prime
2}-rr^{\prime\prime}-r\sqrt{1-{\small r}^{\prime2}}k_{1}\cosh f\right)  ^{2}%
}\\
H_{\overset{part-n}{C_{3}}}=\frac{r\left(  1-{\small r}^{\prime2}\right)
^{3/2}k_{1}\cosh f+3r^{2}\left(  1-{\small r}^{\prime2}\right)  k_{1}^{2}%
\cosh^{2}f-\left(  1-r^{\prime2}-rr^{\prime\prime}\right)  \left(
2-2r^{\prime2}-3rr^{\prime\prime}\right)  }{3r\left(  -r^{2}\left(
1-{\small r}^{\prime2}\right)  k_{1}^{2}\cosh^{2}f+\left(  1-r^{\prime
2}-rr^{\prime\prime}\right)  ^{2}\right)  }%
\end{array}
\right\}  ,\\
\\
\left.
\begin{array}
[c]{l}%
K_{\overset{part-n}{C_{4}}}=-\frac{r\left(  -1+{\small r}^{\prime2}\right)
k_{1}^{2}\sinh^{2}f{\small +}r^{\prime\prime}\left(  -1+r^{\prime2}%
+rr^{\prime\prime}\right)  +\sqrt{-1+{\small r}^{\prime2}}\left(
-1+r^{\prime2}+2rr^{\prime\prime}\right)  k_{1}\sinh f}{r^{2}\left(
-1+r^{\prime2}+rr^{\prime\prime}+r\sqrt{-1+{\small r}^{\prime2}}k_{1}\sinh
f\right)  ^{2}}\\
H_{\overset{part-n}{C_{4}}}=\frac{r\left(  -1+{\small r}^{\prime2}\right)
^{3/2}k_{1}\sinh f-3r^{2}\left(  -1+{\small r}^{\prime2}\right)  k_{1}%
^{2}\sinh^{2}f+\left(  -1+r^{\prime2}+rr^{\prime\prime}\right)  \left(
-2+2r^{\prime2}+3rr^{\prime\prime}\right)  }{3r\left(  r^{2}\left(
-1+{\small r}^{\prime2}\right)  k_{1}^{2}\sinh^{2}f-\left(  -1+r^{\prime
2}+rr^{\prime\prime}\right)  ^{2}\right)  }%
\end{array}
\right\}  ,\\
\\
\left.
\begin{array}
[c]{l}%
K_{\overset{part-n}{C_{5}}}=\frac{r\left(  1+{\small r}^{\prime2}\right)
k_{1}^{2}\sinh^{2}f{\small +}r^{\prime\prime}\left(  1+r^{\prime2}%
+rr^{\prime\prime}\right)  -\sqrt{1+{\small r}^{\prime2}}\left(  1+r^{\prime
2}+2rr^{\prime\prime}\right)  k_{1}\sinh f}{r^{2}\left(  1+r^{\prime
2}+rr^{\prime\prime}-r\sqrt{1+{\small r}^{\prime2}}k_{1}\sinh f\right)  ^{2}%
}\\
H_{\overset{part-n}{C_{5}}}=\frac{r\left(  1+{\small r}^{\prime2}\right)
^{3/2}k_{1}\sinh f+3r^{2}\left(  1+{\small r}^{\prime2}\right)  k_{1}^{2}%
\sinh^{2}f-\left(  1+r^{\prime2}+rr^{\prime\prime}\right)  \left(
2+2r^{\prime2}+3rr^{\prime\prime}\right)  }{3r\left(  r^{2}\left(
1+{\small r}^{\prime2}\right)  k_{1}^{2}\sinh^{2}f-\left(  1+r^{\prime
2}+rr^{\prime\prime}\right)  ^{2}\right)  }%
\end{array}
\right\}  .
\end{array}
\right\}  \label{KHpartn}%
\end{equation}

\end{theorem}

Here, from (\ref{KHpartn}), we can state the following theorem which gives an
important relation between Gaussian and mean curvatures of the canal hypersurfaces:

\begin{proposition}
The Gaussian curvatures and the mean curvatures of canal hypersurfaces
$\overset{part-n}{C_{1}},$ $\overset{part-n}{C_{3}}$ and
$\overset{part-n}{C_{4}}$ given by (\ref{prtn canal1})\ in $E_{1}^{4}$ satisfy%
\begin{equation}
3H_{\overset{part-n}{C_{i}}}-r^{2}K_{\overset{part-n}{C_{i}}}+\frac{2}%
{r}=0,\text{ }i=1,3,4\label{y5y}%
\end{equation}
and the Gaussian curvature and the mean curvature of canal hypersurfaces
$\overset{part-n}{C_{2}}$ and $\overset{part-n}{C_{5}}$ given by
(\ref{prtn canal1}) and (\ref{prtn canal2})\ in $E_{1}^{4}$ satisfy%
\begin{equation}
3H_{\overset{part-n}{C_{i}}}-r^{2}K_{\overset{part-n}{C_{i}}}-\frac{2}%
{r}=0,\text{ }i=2,5.\label{y6y}%
\end{equation}

\end{proposition}

Now, we will give some results for the canal hypersurface
$\overset{part-n}{C_{5}}$ given by (\ref{prtn canal2})\ in $E_{1}^{4}.$
Similarly, one can obtain similar results for the canal hypersurfaces
$\overset{part-n}{C_{1}},$ $\overset{part-n}{C_{2}},$ $\overset{part-n}{C_{3}%
}$ and $\overset{part-n}{C_{4}}$ given by (\ref{prtn canal1})\ in $E_{1}^{4},$ too.

\begin{proposition}
\textit{The canal hypersurface }$\overset{part-n}{C_{5}}$\textit{ that is
formed as the envelope of a family of pseudo hyperbolic hyperspheres whose
centers lie on a partially null curve }$\gamma(s)$\textit{ in }$E_{1}^{4}$ is
flat, if $k_{1}(s)=0$ and the radius function satisfies $r(s)=as+b,$ $a,b\in%
\mathbb{R}
.$
\end{proposition}

\begin{proof}
If we use $k_{1}(s)=0$ in (\ref{KHpartn}), then we have%
\begin{equation}
K_{\overset{part-n}{C_{5}}}=\frac{r^{\prime\prime}}{r^{2}(1+r^{\prime
2}+rr^{\prime\prime})}\label{p1}%
\end{equation}
and this completes the proof.
\end{proof}

\begin{proposition}
\label{prop1}\textit{The canal hypersurface }$\overset{part-n}{C_{5}}$\textit{
that is formed as the envelope of a family of pseudo hyperbolic hyperspheres
whose centers lie on a partially null curve }$\gamma(s)$\textit{ in }%
$E_{1}^{4}$ cannot be flat when $k_{1}(s)\neq0.$
\end{proposition}

\begin{proof}
\textit{Let the canal hypersurface }$\overset{part-n}{C_{5}}$\textit{ be flat.
Then from }(\ref{KHpartn}), we get%
\begin{equation}
r\left(  1+{\small r}^{\prime2}\right)  k_{1}^{2}\sinh^{2}f{\small +}%
r^{\prime\prime}\left(  1+r^{\prime2}+rr^{\prime\prime}\right)  -\sqrt
{1+{\small r}^{\prime2}}\left(  1+r^{\prime2}+2rr^{\prime\prime}\right)
k_{1}\sinh f=0. \label{pp1}%
\end{equation}
Since the set $\{\sinh f,\sinh^{2}f,1\}$ is linearly independent, we get from
(\ref{pp1}) that%
\begin{equation}
r\left(  1+{\small r}^{\prime2}\right)  k_{1}^{2}{\small =}r^{\prime\prime
}\left(  1+r^{\prime2}+rr^{\prime\prime}\right)  =\sqrt{1+{\small r}^{\prime
2}}\left(  1+r^{\prime2}+2rr^{\prime\prime}\right)  k_{1}=0 \label{pp2}%
\end{equation}
and this is a contradiction when $k_{1}(s)\neq0.$ So, this completes the proof.
\end{proof}

\begin{proposition}
\textit{The canal hypersurface }$\overset{part-n}{C_{5}}$\textit{ that is
formed as the envelope of a family of pseudo hyperbolic hyperspheres whose
centers lie on a partially null curve }$\gamma(s)$\textit{ in }$E_{1}^{4}$ is
minimal, if $k_{1}(s)=0$ and the radius function $r(s)$ satisfies $%
{\displaystyle\int}
\frac{dr}{\sqrt{\left(  \frac{a}{r}\right)  ^{\frac{4}{3}}-1}}=\pm s+b,$
$a,b\in%
\mathbb{R}
$.
\end{proposition}

\begin{proof}
If we use $k_{1}(s)=0$ in (\ref{KHpartn}), then we have%
\begin{equation}
H_{\overset{part-n}{C_{5}}}=\frac{2+2r^{\prime2}+3rr^{\prime\prime}%
}{3r(1+r^{\prime2}+rr^{\prime\prime})}.\label{p2}%
\end{equation}
So, if the equation%
\begin{equation}
2+2r^{\prime2}+3rr^{\prime\prime}=0\label{p3}%
\end{equation}
holds, then the canal hypersurface $\overset{part-n}{C_{5}}$ is minimal.

Now, let us solve the equation (\ref{p3}).

If we take $r^{\prime}{}(s)=h(s),$ we get
\begin{equation}
r^{\prime\prime}=h^{\prime}=\frac{dh}{dr}\frac{dr}{ds}=\frac{dh}{dr}h.
\label{min5}%
\end{equation}
Using (\ref{min5}) in (\ref{p3}), we have%
\begin{equation}
3r\frac{dh}{dr}h+2h^{2}+2=0. \label{min6}%
\end{equation}
From (\ref{p3}), $r^{\prime}{}(s)=h(s)\neq0$ and so we reach that%
\begin{equation}
\frac{-3h}{2({\small 1+}h^{2})}dh=\frac{dr}{r}. \label{min7}%
\end{equation}
By integrating (\ref{min7}), we have%
\begin{equation}
\text{ }h=\pm\sqrt{\left(  \frac{a}{r}\right)  ^{\frac{4}{3}}-1}, \label{min8}%
\end{equation}
where $a$ is constant. Since $r^{\prime}{}=\frac{dr}{ds}=h,$ from (\ref{min8})
we get%
\begin{equation}%
{\displaystyle\int}
\frac{dr}{\sqrt{\left(  \frac{a}{r}\right)  ^{\frac{4}{3}}-1}}=\pm%
{\displaystyle\int}
ds \label{min9y}%
\end{equation}

and this completes the proof.
\end{proof}

\begin{proposition}
\textit{The canal hypersurface }$\overset{part-n}{C_{5}}$\textit{ that is
formed as the envelope of a family of pseudo hyperbolic hyperspheres whose
centers lie on a partially null curve }$\gamma(s)$\textit{ in }$E_{1}^{4}$
cannot be minimal when $k_{1}(s)\neq0.$
\end{proposition}

\begin{proof}
\textit{The proof can be done with similar method used in the proof of
Proposition \ref{prop1}.}
\end{proof}

Here, let us construct an example for the canal hypersurface
$\overset{part-n}{C_{5}}(s,t,w)$ that is formed as the envelope of a family of
pseudo hyperspheres whose center lie on a partially null curve in $E_{1}^{4}.$
One can construct the canal hypersurfaces $\overset{part-n}{C_{1}},$
$\overset{part-n}{C_{2}},$ $\overset{part-n}{C_{3}}$ and
$\overset{part-n}{C_{4}},$ similarly.

\begin{example}
Let us take the partially null curve (given in \cite{Kazim2})
\begin{equation}
{\small \gamma(s)=}\frac{1}{2}\left(  2e^{s},2e^{s},\cos(2s),\sin(2s)\right)
\label{ex1yy}%
\end{equation}
in $E_{1}^{4}$. The Frenet vectors and curvatures of the curve (\ref{ex1y})
are%
\begin{equation}
\left.
\begin{array}
[c]{l}%
F_{1}=\left(  e^{s},e^{s},-\sin(2s),\cos(2s)\right)  ,\\
F_{2}=\frac{1}{2}\left(  e^{s},e^{s},-2\cos(2s),-2\sin(2s)\right)  ,\\
F_{3}=\frac{5}{2}\left(  1,1,0,0\right)  ,\\
F_{4}=\left(  -\frac{e^{2s}}{4}-\frac{1}{5},-\frac{e^{2s}}{4}+\frac{1}%
{5},\frac{e^{s}}{5}\left(  \cos(2s)+2\sin(2s)\right)  ,\frac{e^{s}}{5}\left(
\sin(2s)-2\cos(2s)\right)  \right)  ,\\
k_{1}=2,\text{ }k_{2}=e^{s},\text{ }k_{3}=0.
\end{array}
\right\}  \label{ex2yy}%
\end{equation}
If we assume that $g(t,w)=t$ and $f(t,w)=w$ in (\ref{prtn canal2}), the canal
hypersurfaces $\overset{part-n}{C_{5}}(s,t,w)$ can be parametrized by%
\begin{align}
&  \overset{part-n}{C_{5}}(s,t,w)=\label{ex3yy}\\
&  \left(
\begin{array}
[c]{l}%
\frac{s\left(  4+5e^{2s}+100t^{2}\right)  \cosh w}{32\sqrt{5}t}+\frac{e^{s}%
}{8}\left(  8+2s+\sqrt{5}s\sinh w\right)  ,\\
\frac{s\left(  -4+5e^{2s}+100t^{2}\right)  \cosh w}{32\sqrt{5}t}+\frac{e^{s}%
}{8}\left(  8+2s+\sqrt{5}s\sinh w\right)  ,\\
\frac{1}{4}\left(  2\cos(2s)-s\sin(2s)+\sqrt{5}s\left(  -\frac{e^{s}\cosh
w\left(  \cos(2s)+2\sin(2s)\right)  }{10t}-\cos(2s)\sinh w\right)  \right)
,\\
\frac{1}{4}\left(  2\sin(2s)+s\cos(2s)+\sqrt{5}s\left(  -\frac{e^{s}\cosh
w\left(  \sin(2s)-2\cos(2s)\right)  }{10t}-\sin(2s)\sinh w\right)  \right)
\end{array}
\right)  \nonumber
\end{align}
where the radius function has been taken as $r(s)=\frac{s}{2}.$ From
(\ref{KHpartn}), the Gaussian and mean curvatures of the canal hypersurfaces
$\overset{part-n}{C_{5}}$ are obtained as%
\begin{equation}
\left.
\begin{array}
[c]{l}%
K_{\overset{part-n}{C_{5}}}=-\frac{16\sinh w}{s^{2}\left(  \sqrt{5}-2s\sinh
w\right)  },\\
H_{\overset{psd-n}{C_{5}}}=-\frac{20-4s\left(  \sqrt{5}+6s\sinh w\right)
\sinh w}{3s\left(  -5+4s^{2}\sinh^{2}w\right)  }.
\end{array}
\right\}  \label{ex5yy}%
\end{equation}
In Figure 1 (b), one can see the projection of the canal hypersurfaces
(\ref{ex3yy}) for $w=\frac{\pi}{3}$ into $x_{1}x_{3}x_{4}$-space.
\end{example}

\begin{theorem}
\textit{The canal hypersurfaces that are formed as the envelope of a family of
pseudo hyperspheres whose centers lie on a null curve }$\gamma(s)$ with Frenet
vector fields $F_{i},$ $i\in\{1,2,3,4\},$\textit{ in }$E_{1}^{4}$ \textit{can
be parametrized by}%
\begin{equation}
\overset{null}{C_{1}}(s,t,w)=\gamma(s)+a_{1}(s,t,w)F_{1}(s)+a_{2}%
(s,t,w)F_{2}(s)-r(s)r^{\prime}(s)F_{3}(s)+a_{4}(s,t,w)F_{4}(s),
\label{canal1'}%
\end{equation}
where
\[
a_{2}(s,t,w)^{2}+a_{4}(s,t,w)^{2}=r(s)\left(  r(s)+2a_{1}(s,t,w)r^{\prime
}(s)\right)  .
\]

\textit{Furthermore, the canal hypersurfaces that are formed as the envelope
of a family of pseudo hyperbolic hyperspheres whose centers lie on a null
curve }$\gamma(s)$ with Frenet vector fields $F_{i},$ $i\in\{1,2,3,4\},$
\textit{in }$E_{1}^{4}$ \textit{can be parametrized by}%
\begin{equation}
\overset{null}{C_{2}}(s,t,w)=\gamma(s)+a_{1}(s,t,w)F_{1}(s)+a_{2}%
(s,t,w)F_{2}(s)+r(s)r^{\prime}(s)F_{3}(s)+a_{4}(s,t,w)F_{4}(s).
\label{canal2'}%
\end{equation}
where
\[
a_{2}(s,t,w)^{2}+a_{4}(s,t,w)^{2}=-r(s)\left(  r(s)+2a_{1}(s,t,w)r^{\prime
}(s)\right)  .
\]

\end{theorem}

\begin{proof}
Let the center curve $\gamma:I\subseteq%
\mathbb{R}
\rightarrow E_{1}^{4}$ be a null curve with non-zero curvature with Frenet
vector fields. Then, the parametrization of the envelope of pseudo
hyperspheres (or pseudo hyperbolic hyperspheres) defining the canal
hypersurfaces $\overset{null}{C}(s,t,w)$ in $E_{1}^{4}$ can be given by%
\begin{equation}
\overset{null}{C}(s,t,w)-\gamma(s)=a_{1}(s,t,w)F_{1}(s)+a_{2}(s,t,w)F_{2}%
(s)+a_{3}(s,t,w)F_{3}(s)+a_{4}(s,t,w)F_{4}(s). \label{1'}%
\end{equation}
Furthermore, since $\overset{null}{X}(s,t,w)$ lies on the pseudo hyperspheres
($\lambda=1$) (or pseudo hyperbolic hyperspheres ($\lambda=-1$)), we have%
\begin{equation}
g(\overset{null}{C}(s,t,w)-\gamma(s),\overset{null}{C}(s,t,w)-\gamma
(s))=\lambda r^{2}(s) \label{2'}%
\end{equation}
which leads to from (\ref{1'}) and (\ref{nullFi ler}) that%
\begin{equation}
a_{2}^{2}+a_{4}^{2}+2a_{1}a_{3}=\lambda r^{2} \label{3'}%
\end{equation}
and%
\begin{equation}
a_{2}a_{2_{s}}+a_{4}a_{4_{s}}+a_{1}a_{3_{s}}+a_{3}a_{1_{s}}=\lambda rr_{s}.
\label{4'}%
\end{equation}

So, differentiating (\ref{1'}) with respect to $s$ and using the Frenet
formula (\ref{nullFrenet}), we get%
\begin{align}
(\overset{null}{C})_{s}  &  =\left(  1+a_{2}k_{2}-a_{4}k_{3}+a_{1_{s}}\right)
F_{1}+\left(  a_{1}k_{1}-a_{3}k_{2}+a_{2_{s}}\right)  F_{2}\label{5'y}\\
&  +\left(  -a_{2}k_{1}+a_{3_{s}}\right)  F_{3}+\left(  a_{3}k_{3}+a_{4_{s}%
}\right)  F_{4}.\nonumber
\end{align}
Furthermore, $\overset{null}{C}(s,t,w)-\gamma(s)$ is a normal vector to the
canal hypersurfaces, which implies that%
\begin{equation}
g(\overset{null}{C}(s,t,w)-\gamma(s),(\overset{null}{C})_{s}(s,t,w))=0
\label{6'}%
\end{equation}
and so, from (\ref{1'}), (\ref{5'y}) and (\ref{6'}) we have%
\begin{equation}
\text{ }\left(
\begin{array}
[c]{l}%
a_{1}\left(  -a_{2}k_{1}+a_{3_{s}}\right)  +a_{2}\left(  a_{1}k_{1}-a_{3}%
k_{2}+a_{2_{s}}\right) \\
+a_{3}\left(  1+a_{2}k_{2}-a_{4}k_{3}+a_{1_{s}}\right)  +a_{4}\left(
a_{3}k_{3}+a_{4_{s}}\right)
\end{array}
\right)  =0. \label{6'y}%
\end{equation}
Using (\ref{4'}) in (\ref{6'y}), we get
\begin{equation}
a_{3}=-\lambda rr_{s}. \label{7'}%
\end{equation}
Hence, using (\ref{7'}) in (\ref{3'}), we reach that%
\begin{equation}
{\small a}_{2}^{2}{\small +}a_{4}^{2}{\small =\lambda r(}r+2a_{1}%
{\small r}_{s}{\small ).} \label{8'}%
\end{equation}
Therefore from (\ref{7'}) and (\ref{8'}), the canal hypersurfaces
$\overset{null}{C}(s,t,w)$ that are formed as the envelope of a family of
pseudo hyperspheres or pseudo hyperbolic hyperspheres\textit{ }whose centers
lie on a null curve in $E_{1}^{4}$ can be parametrized by (\ref{canal1'}) or
(\ref{canal2'}), respectively.
\end{proof}

Here, let us construct an example for the canal hypersurface
$\overset{null}{C_{1}}(s,t,w)$ that is formed as the envelope of a family of
pseudo hyperspheres whose center lie on a null curve in $E_{1}^{4}.$ One can
construct the canal hypersurface $\overset{null}{C_{2}},$ similarly.

\begin{example}
Let us take the null curve (given in \cite{Kazim3})
\begin{equation}
{\small \gamma(s)=}\frac{1}{\sqrt{2}}\left(  \sinh s,\cosh s,\sin s,\cos
s\right)  \label{ex1yyy}%
\end{equation}
in $E_{1}^{4}$. The Frenet vectors and curvatures of the curve (\ref{ex1y})
are%
\begin{equation}
\left.
\begin{array}
[c]{l}%
F_{1}=\frac{1}{\sqrt{2}}\left(  \cosh s,\sinh s,\cos s,-\sin s\right)  ,\\
F_{2}=\frac{1}{\sqrt{2}}\left(  \sinh s,\cosh s,-\sin s,-\cos s\right)  ,\\
F_{3}=\frac{1}{\sqrt{2}}\left(  -\cosh s,-\sinh s,\cos s,-\sin s\right)  ,\\
F_{4}=\frac{1}{\sqrt{2}}\left(  \sinh s,\cosh s,\sin s,\cos s\right)  ,\\
k_{1}=1,\text{ }k_{2}=0,\text{ }k_{3}=-1.
\end{array}
\right\}  \label{ex2yyy}%
\end{equation}
If we assume that $g(t,w)=t$ and $f(t,w)=w$ in (\ref{prtn canal2}), the canal
hypersurfaces $\overset{null}{C_{1}}(s,t,w)$ can be parametrized by%
\begin{align}
&  \overset{null}{C_{1}}(s,t,w)=\label{ex3yyy}\\
&  \frac{1}{8\sqrt{2}t}\left(
\begin{array}
[c]{l}%
-2st\left(  1+\sqrt{3}\cos w\right)  \cosh s+\left(  8t+\sqrt{3}s\left(
1+2t^{2}\right)  \sin w\right)  \sinh s,\\
-2st\left(  1+\sqrt{3}\cos w\right)  \sinh s+\left(  8t+\sqrt{3}s\left(
1+2t^{2}\right)  \sin w\right)  \cosh s,\\
-2st\left(  1-\sqrt{3}\cos w\right)  \cos s+\left(  8t+\sqrt{3}s\left(
1-2t^{2}\right)  \sin w\right)  \sin s,\\
-2st\left(  -1+\sqrt{3}\cos w\right)  \sin s+\left(  8t+\sqrt{3}s\left(
1-2t^{2}\right)  \sin w\right)  \cos s
\end{array}
\right) \nonumber
\end{align}
where the radius function has been taken as $r(s)=\frac{s}{2}.$

In Figure 1 (c), one can see the projection of the canal hypersurfaces
(\ref{ex3yyy}) for $w=\frac{\pi}{3}$ into $x_{1}x_{3}x_{4}$-space.
\end{example}

\begin{figure}[H]
\centering
\includegraphics[
height=2.4in, width=6.4in
]{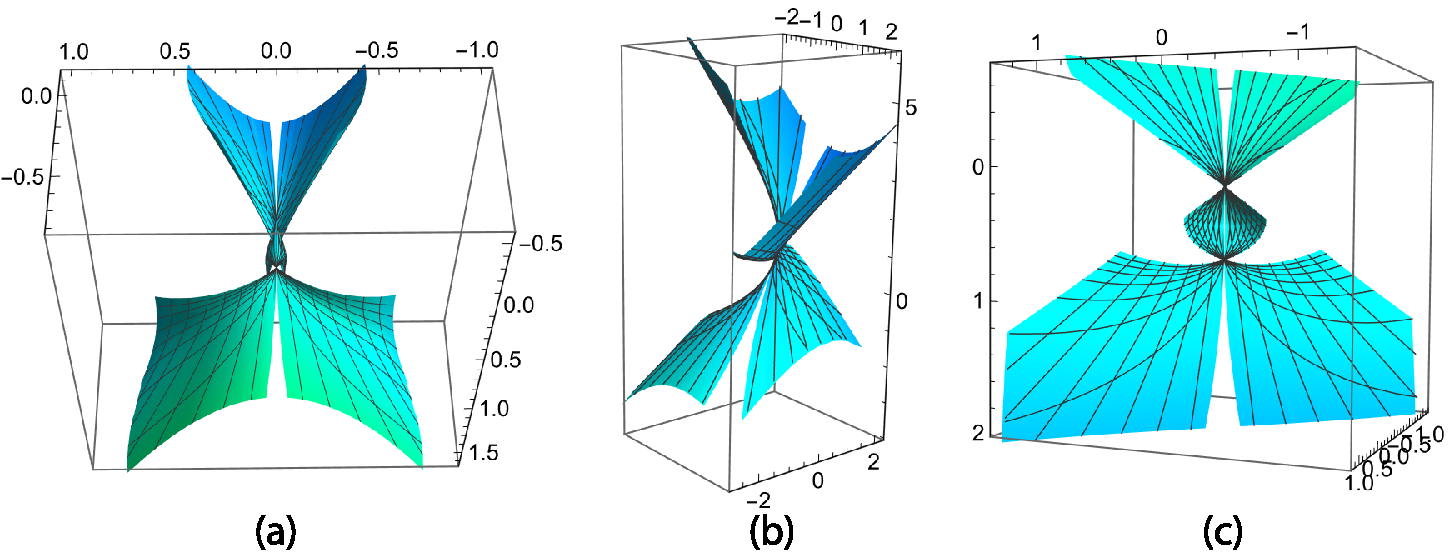}\caption{ }%
\label{fig:1}%
\end{figure}

\section{\textbf{TUBULAR HYPERSURFACES GENERATED BY PSEUDO-NULL, PARTIALLY
NULL AND NULL CURVES IN }$E_{1}^{4}$}

In this section, by taking the radius function $r(s)=r$ is constant in the
Section 2, firstly we give the paremetric expressions of the tubular
hypersurfaces which are formed as the envelope of a family of pseudo
hyperspheres or pseudo hyperbolic hyperspheres whose centers lie on a
pseudo-null curve, partially null curve or null curve in $E_{1}^{4}$. After
that, we obtain some results for Weingarten tubular hypersurfaces with the aid
of the Gaussian curvatures and mean curvatures of these tubular hypersurfaces.

\begin{theorem}
\textit{The tubular hypersurfaces that are formed as the envelope of a family
of pseudo hyperspheres whose centers lie on pseudo null curves with }Frenet
vector fields $F_{i},$ $i\in\{1,2,3,4\},$\textit{ in }$E_{1}^{4}$ \textit{can
be parametrized by}%
\begin{equation}
\left.
\begin{array}
[c]{l}%
\overset{psd-n}{T_{1}}(s,t,w)=\gamma(s)\mp r\left(  g(t,w)\sin(f(t,w))F_{2}%
(s)+\cos(f(t,w))F_{3}(s)+\frac{\sin(f(t,w))}{2g(t,w)}F_{4}(s)\right)  ,\\
\overset{psd-n}{T_{2}}(s,t,w)=\gamma(s)\mp r\left(  g(t,w)\cos(f(t,w))F_{2}%
(s)+\sin(f(t,w))F_{3}(s)+\frac{\cos(f(t,w))}{2g(t,w)}F_{4}(s)\right)  ,\\
\overset{psd-n}{T_{3}}(s,t,w)=\gamma(s)\mp r\left(  g(t,w)\sinh(f(t,w))F_{2}%
(s)+\cosh(f(t,w))F_{3}(s)-\frac{\sinh(f(t,w))}{2g(t,w)}F_{4}(s)\right)  .
\end{array}
\right\}  \label{t1}%
\end{equation}

\textit{Furthermore, the tubular hypersurfaces that are formed as the envelope
of a family of pseudo hyperbolic hyperspheres whose centers lie on pseudo null
curves with }Frenet vector fields $F_{i},$ $i\in\{1,2,3,4\},$\textit{ in
}$E_{1}^{4}$ \textit{can be parametrized by}%
\begin{equation}
\overset{psd-n}{T_{4}}(s,t,w)=\gamma(s)\mp r\left(  g(t,w)\cosh(f(t,w))F_{2}%
(s)+\sinh(f(t,w))F_{3}(s)-\frac{\cosh(f(t,w))}{2g(t,w)}F_{4}(s)\right)  .
\label{t2}%
\end{equation}

\end{theorem}

\begin{theorem}
\textit{The Gaussian and mean curvatures of the tubular hypersurfaces
}$\overset{psd-n}{T_{i}},$ $i\in\{1,2,3,4\},$ given by (\ref{t1}) and
(\ref{t2}) in $E_{1}^{4}$ are%
\begin{equation}
\left.
\begin{array}
[c]{l}%
K_{\overset{psd-n}{T_{1}}}=\frac{k_{1}}{r^{2}\left(  2g\csc f-rk_{1}\right)
},\text{ \ \ }H_{\overset{psd-n}{T_{1}}}=\frac{1}{-r+\frac{2rg}{-4g+3rk_{1}%
\sin f}};\\
K_{\overset{psd-n}{T_{2}}}=\frac{k_{1}}{r^{2}\left(  rk_{1}-2g\sec f\right)
},\text{ \ \ }H_{\overset{psd-n}{T_{2}}}=\frac{1}{r+\frac{2rg}{4g-3rk_{1}\cos
f}};\\
K_{\overset{psd-n}{T_{3}}}=\frac{-k_{1}}{r^{2}\left(  2g\text{csch}%
f+rk_{1}\right)  },\text{ \ \ }H_{\overset{psd-n}{T_{3}}}=\frac{1}%
{-r+\frac{2rg}{-4g-3rk_{1}\sinh f}};\\
K_{\overset{psd-n}{T_{4}}}=\frac{k_{1}}{r^{2}\left(  2g\text{sech}%
f+rk_{1}\right)  },\text{ \ \ }H_{\overset{psd-n}{T_{4}}}=\frac{1}%
{r+\frac{2rg}{4g+3rk_{1}\cosh f}}.
\end{array}
\right\}  \label{t3}%
\end{equation}

\end{theorem}

So from (\ref{t3}), we can give the following results:

\begin{proposition}
\textbf{i)} If the pseudo null curve $\gamma(s),$ which generates the tubular
hypersurface $\overset{psd-n}{T_{1}}$ given by (\ref{t1})\ in $E_{1}^{4},$ is
a straight line, then the tubular hypersurface is flat.

\textbf{ii)} Let the pseudo null curve $\gamma(s),$ which generates the
tubular hypersurface $\overset{psd-n}{T_{1}}$ given by (\ref{t1})\ in
$E_{1}^{4},$ not be a straight line. If $g(t,w)=\sin(f(t,w))$, then the
Gaussian curvature is constant with $\frac{1}{r^{2}(2-r)}.$
\end{proposition}

\begin{proposition}
\textbf{i)} The tubular hypersurface $\overset{psd-n}{T_{1}}$ given by
(\ref{t1})\ in $E_{1}^{4}$ cannot be minimal.

\textbf{ii) }If the pseudo null curve $\gamma(s),$ which generates the tubular
hypersurface $\overset{psd-n}{T_{1}}$ given by (\ref{t1})\ in $E_{1}^{4},$ is
a straight line, then the mean curvature of the tubular hypersurface
$\overset{psd-n}{T_{1}}$ given by (\ref{t1})\ in $E_{1}^{4}$ is constant with
$\frac{-2}{3r}$.

\textbf{iii)} Let the pseudo null curve $\gamma(s),$ which generates the
tubular hypersurface $\overset{psd-n}{T_{1}}$ given by (\ref{t1})\ in
$E_{1}^{4},$ not be a straight line. If $g(t,w)=\sin(f(t,w)),$ then the mean
curvature of the tubular hypersurface $\overset{psd-n}{T_{1}}$ given by
(\ref{t1})\ in $E_{1}^{4}$ is constant with $\frac{3r-4}{3r(2-r)},$ $2\neq
r\neq\frac{4}{3}$.
\end{proposition}

Also, if%
\begin{equation}
\left.
\begin{array}
[c]{l}%
\left(  H_{\overset{psd-n}{T_{i}}}\right)  _{s}\left(
K_{\overset{psd-n}{T_{i}}}\right)  _{t}-\left(  H_{\overset{psd-n}{T_{i}}%
}\right)  _{t}\left(  K_{\overset{psd-n}{T_{i}}}\right)  _{s}=0,\\
\left(  H_{\overset{psd-n}{T_{i}}}\right)  _{s}\left(
K_{\overset{psd-n}{T_{i}}}\right)  _{w}-\left(  H_{\overset{psd-n}{T_{i}}%
}\right)  _{w}\left(  K_{\overset{psd-n}{T_{i}}}\right)  _{s}=0,\\
\left(  H_{\overset{psd-n}{T_{i}}}\right)  _{t}\left(
K_{\overset{psd-n}{T_{i}}}\right)  _{w}-\left(  H_{\overset{psd-n}{T_{i}}%
}\right)  _{w}\left(  K_{\overset{psd-n}{T_{i}}}\right)  _{t}=0
\end{array}
\right\}  \label{t4}%
\end{equation}
hold on a hypersurface, then we call the hypersurface as $\left(
H_{\overset{psd-n}{T_{i}}},K_{\overset{psd-n}{T_{i}}}\right)  _{st}%
$-Weingarten, $\left(  H_{\overset{psd-n}{T_{i}}},K_{\overset{psd-n}{T_{i}}%
}\right)  _{sw}$-Weingarten, $\left(  H_{\overset{psd-n}{T_{i}}}%
,K_{\overset{psd-n}{T_{i}}}\right)  _{tw}$-Weingarten hypersurface,
respectively, where $\left(  H_{\overset{psd-n}{T_{i}}}\right)  _{s}%
=\frac{\partial\left(  H_{\overset{psd-n}{T_{i}}}\right)  }{\partial s}$ and
so on. Thus,

\begin{theorem}
The tubular hypersurfaces $\overset{psd-n}{T_{i}},$ $i\in\{1,2,3,4\},$ given
by (\ref{t1}) and (\ref{t2}) in $E_{1}^{4},$ are $\left(
H_{\overset{psd-n}{T_{i}}},K_{\overset{psd-n}{T_{i}}}\right)  _{st}%
$-Weingarten, $\left(  H_{\overset{psd-n}{T_{i}}},K_{\overset{psd-n}{T_{i}}%
}\right)  _{sw}$-Weingarten and $\left(  H_{\overset{psd-n}{T_{i}}%
},K_{\overset{psd-n}{T_{i}}}\right)  _{tw}$-Weingarten.
\end{theorem}

\begin{proof}
Let us give the proof for $i=1.$ From (\ref{t3}), we have%
\begin{equation}%
\begin{array}
[c]{l}%
\left(  H_{\overset{psd-n}{T_{1}}}\right)  _{s}=\frac{2k_{1}^{\prime}g\sin
f}{3\left(  -2g+rk_{1}\sin f\right)  ^{2}}\\
\left(  H_{\overset{psd-n}{T_{1}}}\right)  _{t}=\frac{2k_{1}\left(  gf_{t}\cos
f-g_{t}\sin f\right)  }{3\left(  -2g+rk_{1}\sin f\right)  ^{2}}\\
\left(  H_{\overset{psd-n}{T_{1}}}\right)  _{w}=\frac{2k_{1}\left(  gf_{w}\cos
f-g_{w}\sin f\right)  }{3\left(  -2g+rk_{1}\sin f\right)  ^{2}}%
\end{array}
\text{ \ \ \ and \ \ \ }\left.
\begin{array}
[c]{l}%
\left(  K_{\overset{psd-n}{T_{1}}}\right)  _{s}=\frac{2k_{1}^{\prime}g\sin
f}{r^{2}\left(  -2g+rk_{1}\sin f\right)  ^{2}}\\
\left(  K_{\overset{psd-n}{T_{1}}}\right)  _{t}=\frac{2k_{1}\left(  gf_{t}\cos
f-g_{t}\sin f\right)  }{r^{2}\left(  -2g+rk_{1}\sin f\right)  ^{2}}\\
\left(  K_{\overset{psd-n}{T_{1}}}\right)  _{w}=\frac{2k_{1}\left(  gf_{w}\cos
f-g_{w}\sin f\right)  }{r^{2}\left(  -2g+rk_{1}\sin f\right)  ^{2}},
\end{array}
\right\}  \label{t5}%
\end{equation}
where $f_{t}=\frac{\partial f(t,w)}{\partial t},$ $f_{w}=\frac{\partial
f(t,w)}{\partial w}$ and so on.

So, from (\ref{t5}) we get%
\begin{equation}
\left.
\begin{array}
[c]{l}%
\left(  H_{\overset{psd-n}{T_{1}}}\right)  _{s}\left(
K_{\overset{psd-n}{T_{1}}}\right)  _{t}-\left(  H_{\overset{psd-n}{T_{1}}%
}\right)  _{t}\left(  K_{\overset{psd-n}{T_{1}}}\right)  _{s}=0,\\
\left(  H_{\overset{psd-n}{T_{1}}}\right)  _{s}\left(
K_{\overset{psd-n}{T_{1}}}\right)  _{w}-\left(  H_{\overset{psd-n}{T_{1}}%
}\right)  _{w}\left(  K_{\overset{psd-n}{T_{1}}}\right)  _{s}=0,\\
\left(  H_{\overset{psd-n}{T_{1}}}\right)  _{t}\left(
K_{\overset{psd-n}{T_{1}}}\right)  _{w}-\left(  H_{\overset{psd-n}{T_{1}}%
}\right)  _{w}\left(  K_{\overset{psd-n}{T_{1}}}\right)  _{t}=0
\end{array}
\right\}  \label{t7}%
\end{equation}
and this completes the proof. Similarly, one can obtain the results for
$i=2,3,4$.
\end{proof}

\begin{theorem}
\textit{The tubular hypersurfaces that are formed as the envelope of a family
of pseudo hyperspheres whose centers lie on partially null curves with }Frenet
vector fields $F_{i},$ $i\in\{1,2,3,4\},$\textit{ in }$E_{1}^{4}$ \textit{can
be parametrized by}%
\begin{equation}
\left.
\begin{array}
[c]{l}%
\overset{part-n}{T_{1}}(s,t,w)=\gamma(s)\mp r\left(  \sin(f(t,w))F_{2}%
(s)+g(t,w)\cos(f(t,w))F_{3}(s)+\frac{\cos(f(t,w))}{2g(t,w)}F_{4}(s)\right)
,\\
\overset{part-n}{T_{2}}(s,t,w)=\gamma(s)\mp r\left(  \cos(f(t,w))F_{2}%
(s)+g(t,w)\sin(f(t,w))F_{3}(s)+\frac{\sin(f(t,w))}{2g(t,w)}F_{4}(s)\right)
,\\
\overset{part-n}{T_{3}}(s,t,w)=\gamma(s)\mp r\left(  \cosh(f(t,w))F_{2}%
(s)+g(t,w)\sinh(f(t,w))F_{3}(s)-\frac{\sinh(f(t,w))}{2g(t,w)}F_{4}(s)\right)
,
\end{array}
\right\}  \label{t8}%
\end{equation}

\textit{Furthermore, the tubular hypersurfaces that are formed as the envelope
of a family of pseudo hyperbolic hyperspheres whose centers lie on partially
null curves with }Frenet vector fields $F_{i},$ $i\in\{1,2,3,4\},$\textit{ in
}$E_{1}^{4}$ \textit{can be parametrized by}%
\begin{equation}
\overset{part-n}{T_{4}}(s,t,w)=\gamma(s)\mp r\left(  \sinh(f(t,w))F_{2}%
(s)+g(t,w)\cosh(f(t,w))F_{3}(s)-\frac{\cosh(f(t,w))}{2g(t,w)}F_{4}(s)\right)
. \label{t9}%
\end{equation}

\end{theorem}

\begin{theorem}
\textit{The Gaussian and mean curvatures of the tubular hypersurfaces
}$\overset{part-n}{T_{i}},$ $i\in\{1,2,3,4\},$ given by (\ref{t8}) and
(\ref{t9}) in $E_{1}^{4}$ are%
\begin{equation}
\left.
\begin{array}
[c]{l}%
K_{\overset{part-n}{T_{1}}}=\frac{k_{1}}{r^{2}\left(  \csc f-rk_{1}\right)
},\text{ \ \ }H_{\overset{part-n}{T_{1}}}=\frac{1}{-r+\frac{r}{-2+3rk_{1}\sin
f}};\\
K_{\overset{part-n}{T_{2}}}=\frac{-k_{1}}{r^{2}\left(  \sec f-rk_{1}\right)
},\text{ \ \ }H_{\overset{part-n}{T_{2}}}=\frac{-1}{-r+\frac{r}{-2+3rk_{1}\cos
f}};\\
K_{\overset{part-n}{T_{3}}}=\frac{k_{1}}{r^{2}\left(  \text{sech}%
f-rk_{1}\right)  },\text{ \ \ }H_{\overset{part-n}{T_{3}}}=\frac{1}%
{-r+\frac{r}{-2+3rk_{1}\cosh f}};\\
K_{\overset{part-n}{T_{4}}}=\frac{-k_{1}}{r^{2}\left(  \text{csch}%
f-rk_{1}\right)  },\text{ \ \ }H_{\overset{part-n}{T_{4}}}=\frac{-1}%
{-r+\frac{r}{-2+3rk_{1}\sinh f}}.
\end{array}
\right\}  \label{t10}%
\end{equation}

\end{theorem}

Thus,

\begin{proposition}
If $k_{1}(s)=0$ for the partially null curve $\gamma(s),$ which
generates the tubular hypersurfaces $\overset{part-n}{T_{i}}$ given
by (\ref{t8})\ and (\ref{t9})\ in $E_{1}^{4}$, then these tubular
hypersurfaces are flat and also, the tubular hypersurfaces
$\overset{part-n}{T_{i}}$ have constant mean curvature with
$\frac{2(-1)^{i}}{3r}$.
\end{proposition}

\begin{theorem}
The tubular hypersurfaces $\overset{part-n}{T_{i}},$ $i\in\{1,2,3,4\},$ given
by (\ref{t8}) and (\ref{t9}) in $E_{1}^{4},$ are $\left(
H_{\overset{part-n}{T_{i}}},K_{\overset{part-n}{T_{i}}}\right)  _{st}%
$-Weingarten, $\left(  H_{\overset{part-n}{T_{i}}},K_{\overset{part-n}{T_{i}}%
}\right)  _{sw}$-Weingarten and $\left(  H_{\overset{part-n}{T_{i}}%
},K_{\overset{part-n}{T_{i}}}\right)  _{tw}$-Weingarten.
\end{theorem}

\begin{proof}
Here we will give the proof for $i=3.$ From (\ref{t10}), we have%
\begin{equation}%
\begin{array}
[c]{l}%
\left(  H_{\overset{part-n}{T_{3}}}\right)  _{s}=\frac{k_{1}^{\prime}\cosh
f}{3\left(  -1+rk_{1}\cosh f\right)  ^{2}}\\
\left(  H_{\overset{part-n}{T_{3}}}\right)  _{t}=\frac{k_{1}f_{t}\sinh
f}{3\left(  -1+rk_{1}\cosh f\right)  ^{2}}\\
\left(  H_{\overset{part-n}{T_{3}}}\right)  _{w}=\frac{k_{1}f_{w}\sinh
f}{3\left(  -1+rk_{1}\cosh f\right)  ^{2}}%
\end{array}
\text{ \ \ \ and \ \ \ }\left.
\begin{array}
[c]{l}%
\left(  K_{\overset{part-n}{T_{3}}}\right)  _{s}=\frac{k_{1}^{\prime}\cosh
f}{r^{2}\left(  -1+rk_{1}\cosh f\right)  ^{2}}\\
\left(  K_{\overset{part-n}{T_{3}}}\right)  _{t}=\frac{k_{1}f_{t}\sinh
f}{r^{2}\left(  -1+rk_{1}\cosh f\right)  ^{2}}\\
\left(  K_{\overset{part-n}{T_{3}}}\right)  _{w}=\frac{k_{1}f_{w}\sinh
f}{r^{2}\left(  -1+rk_{1}\cosh f\right)  ^{2}}.
\end{array}
\right\}  \label{t11}%
\end{equation}

So, from (\ref{t11}) we get%
\begin{equation}
\left.
\begin{array}
[c]{l}%
\left(  H_{\overset{part-n}{T_{3}}}\right)  _{s}\left(
K_{\overset{part-n}{T_{3}}}\right)  _{t}-\left(  H_{\overset{part-n}{T_{3}}%
}\right)  _{t}\left(  K_{\overset{part-n}{T_{3}}}\right)  _{s}=0,\\
\left(  H_{\overset{part-n}{T_{3}}}\right)  _{s}\left(
K_{\overset{part-n}{T_{3}}}\right)  _{w}-\left(  H_{\overset{part-n}{T_{3}}%
}\right)  _{w}\left(  K_{\overset{part-n}{T_{3}}}\right)  _{s}=0,\\
\left(  H_{\overset{part-n}{T_{3}}}\right)  _{t}\left(
K_{\overset{part-n}{T_{3}}}\right)  _{w}-\left(  H_{\overset{part-n}{T_{3}}%
}\right)  _{w}\left(  K_{\overset{part-n}{T_{3}}}\right)  _{t}=0
\end{array}
\right\}  \label{y7y}%
\end{equation}
and this completes the proof. Similarly, one can obtain the results for
$i=1,2,4$.
\end{proof}

\begin{theorem}
\textit{The tubular hypersurfaces that are formed as the envelope of a family
of pseudo hyperspheres whose centers lie on a null curve }$\gamma(s)$ with
Frenet vector fields $F_{i},$ $i\in\{1,2,3,4\},$\textit{ in }$E_{1}^{4}$
\textit{can be parametrized by}%
\begin{equation}
\overset{null}{T_{1}}(s,t,w)=\gamma(s)+a_{1}(s,t,w)F_{1}(s)+a_{2}%
(s,t,w)F_{2}(s)+a_{4}(s,t,w)F_{4}(s), \label{t13}%
\end{equation}
where
\[
a_{2}(s,t,w)^{2}+a_{4}(s,t,w)^{2}=r^{2}.
\]
\textit{There is no tubular hypersurfaces that are formed as the envelope of a
family of pseudo hyperbolic hyperspheres whose centers lie on a null curve
}$\gamma(s)$ with Frenet vector fields $F_{i},$ $i\in\{1,2,3,4\},$\textit{ in
}$E_{1}^{4}.$
\end{theorem}

\end{document}